\documentclass[11pt]{article}

\usepackage{mathtools}
\usepackage{amssymb, amsfonts,amsthm}
\usepackage[a4paper]{geometry}
\usepackage{lmodern}
\usepackage[T1]{fontenc}
\usepackage[utf8]{inputenc}
\usepackage{soul}
\usepackage{titlesec} 
\usepackage{authblk}
\usepackage{bbm}
\usepackage{graphicx}
\usepackage{stmaryrd} % necessary for \llbracket and \rrbracket
\usepackage{tikz} 
\usepackage{tdsfrmath} %necessary for \interoo
\usetikzlibrary{hobby} %n necessary for tikz 
\usepackage[numbers,sort]{natbib}
\usepackage{fancyhdr}
\usepackage[english]{babel}
\usepackage{hyperref}
\usepackage{subcaption} % necessary for side by side figures
\usepackage[open,level=1]{bookmark}

\fancypagestyle{plain}{%
\fancyhf{}
\fancyfoot[C]{\iffloatpage{}{\thepage}}
}
\NewDocumentCommand{\An}{}{A_n[\infty]}

\NewDocumentCommand{\config}{}{n\mathbbm{1}_{\mathcal{H}}}
\NewDocumentCommand{\Rk}{}{\mathcal{R}_{k\sqrt{\log n}}}
\NewDocumentCommand{\Mk}{}{M_{k\sqrt{\log n}}}
\NewDocumentCommand{\Wk}{}{W_{k\sqrt{\log n}}}

\NewDocumentCommand{\Mnl}{}{M_{n/2+l}}

\NewDocumentCommand{\Mnstarl}{}{M^*_{n/2+l}}
\NewDocumentCommand{\Tauk}{}{\mathcal{T}_{k\sqrt{\log n}}}
\NewDocumentCommand{\init}{}{n\mathbbm{1}_{\mathcal{H}}}
\NewDocumentCommand{\MMk}{}{\widetilde{M}_{k\sqrt{\log n}}}
\NewDocumentCommand{\nsubset}{}{\not\subset}
\NewDocumentCommand{\dx}{}{\mathrm{d}}
\DeclareMathOperator{\Over}{\normalfont{\textbf{Over}}}

\NewDocumentCommand{\Dext}{}{\textbf{D}_{\mathrm{ext}}}
\NewDocumentCommand{\Cmid}{}{\textbf{m}_i}

\newcommand{\B}[1]{\Over(M,\varepsilon,#1)}
\newcommand{\Bbis}[1]{\Over(M,\varepsilon,#1)^c}
\newcommand{\Bn}[1]{\Over(M_{#1},\varepsilon_{#1})}
\newcommand{\Bbisn}[1]{\Over(M_{#1},\varepsilon_{#1})^c}
\newcommand{\comment}[1]{\textcolor{black}{#1}}
\newcommand{\commentK}[1]{\textcolor{black}{#1}}

\numberwithin{equation}{section}
\theoremstyle{plain}
\newtheorem{theorem}{Theorem}[section]
\newtheorem{proposition}[theorem]{Proposition}
\newtheorem{lemma}[theorem]{Lemma}

\theoremstyle{definition}

\title{IDLA with sources in a hyperplane of $\mathbb{Z}^d$}

\author[1]{\textsc{Nicolas Chenavier}}
\author[2]{\textsc{David Coupier}}
\author[1]{\textsc{Keenan Penner}}
\author[3]{\textsc{Arnaud Rousselle}} 
\affil[1]{Université du Littoral Côte d'Opale, UR 2597, LMPA, Laboratoire de Mathématiques Pures et Appliquées Joseph Liouville,
62100 Calais, France.}
\affil[2]{Institut Mines Télécom Nord Europe, Cité Scientifique, 59655 Villeneuve d'Ascq, France.}
\affil[3]{Universit\'e Bourgogne Europe, CNRS, IMB UMR 5584,
F-21000 Dijon, France.}
\date{}

\begin{document}
\maketitle

\begin{abstract}
	We consider a random growth model based on the IDLA protocol with sources in a hyperplane of $\mathbb{Z}^d$. We provide a stabilization result and a shape theorem generalizing \cite{chenavier2023bi} in any dimension by introducing new techniques leading to a rough global upper bound.
\end{abstract}



\section{Introduction}
\label{section: intro}

The (standard) Internal Diffusion Limited Aggregation (IDLA) is a random growth model $(A_n)_{n\geq 0}$ in $\mathbb{Z}^d$ recursively defined as follows. We start with $A_0=\emptyset$. At step $n$, a simple symmetric random walk (independent of everything else) starts from the origin $0$, called the \textit{source}, until it exits the current aggregate $A_{n-1}$, say at some vertex $z$, which is added to $A_{n-1}$ to get $A_n = A_{n-1}\cup\{z\}$. A first shape theorem was established by Lawler, Bramson and Griffeath in \cite{lawler1992internal}. It asserts that the aggregate $A_n$ (when it is suitably normalized) converges a.s.\,to an Euclidean ball as the number $n$ of random walks goes to infinity, with fluctuations (w.r.t.\,the limit shape) which are at most linear. Since then, several papers (by Lawler \cite{lawler95}, Asselah and Gaudilli\`ere \cite{asselah2013logarithmic,asselah2013sublogarithmic,AG14} and Jerison, Levine and Sheffield \cite{jerison2012logarithmic,JLS13,JLS14}) have improved the bounds for fluctuations which are known to be logarithmic in dimension $2$ and sublogarithmic in higher dimensions. Since then, many variants of this model have been considered and corresponding shape theorems have been explored. Let us cite IDLA models on discrete groups with polynomial or exponential growth in \cite{B04,BlB07}, on non-amenable graphs in \cite{H08}, on comb lattices in \cite{AR16,HS12}, on cylinder graphs in \cite{JLS14b,LS,S19} or on supercritical percolation clusters in \cite{DLYY, Shellef}. Let us mention that IDLA models with drifted random walks \cite{L14} or with uniform starting points \cite{BDCKL} have been also studied. The case of multiple sources has been investigated too, see e.g. \comment{\cite{chenavier2023bi,LP10, freiberg2023internal}}. 

In this paper, we aim to extend the shape theorem in dimension $d=2$ stated in \cite{chenavier2023bi} to higher dimensions. As explained below, this generalization is non-trivial and requires new ideas.\\

The infinite set of sources that we consider is the hyperplane $\mathcal{H}:=\{0\} \times \mathbb{Z}^{d-1}$ of $\mathbb{Z}^{d}$, with $d\geq 3$. A random walk starting from a source of $\mathcal{H}$ and stopped when it exits the current aggregate is called a \textit{particle}. Let $n,M$ be non-negative integers. In the sequel, exactly $n$ particles are sent from each source. Let us now build the sequence of aggregates $(A_n[M])_{M\geq 0}$ inductively as follows. When $M=0$, $A_n[0]$ is the classical IDLA model, i.e.\,with $n$ particles emitted from the origin. Let us call \textit{level} $M$ the set of sources in $\mathcal{H}$ at distance $M$ from the origin (for $\|(z_1,\ldots,z_d)\| := \max_i |z_i|$). Given a realization of $A_n[M-1]$, we throw $n$ particles from each source of level $M$ according to the lexicographical order. So $A_n[M]$ is defined as the aggregate produced by $A_n[M-1]$ and the new sites added by particles launched at level $M$.

Unlike its shape, the total number of sites in $A_n[M]$ is deterministic, and equals $\#A_n[M] = n(2M+1)^{d-1}$. Besides, by construction, the sequence of aggregates $(A_n[M])_{M\geq 0}$ is a.s.\,increasing in the sense of inclusion, allowing us to define the limiting aggregate $A_n[\infty]$ as:
\[
A_n[\infty] := \bigcup_{M\geq 0} A_n[M] \quad \text{a.s.}
\]

\begin{figure}
	\centering
	\includegraphics{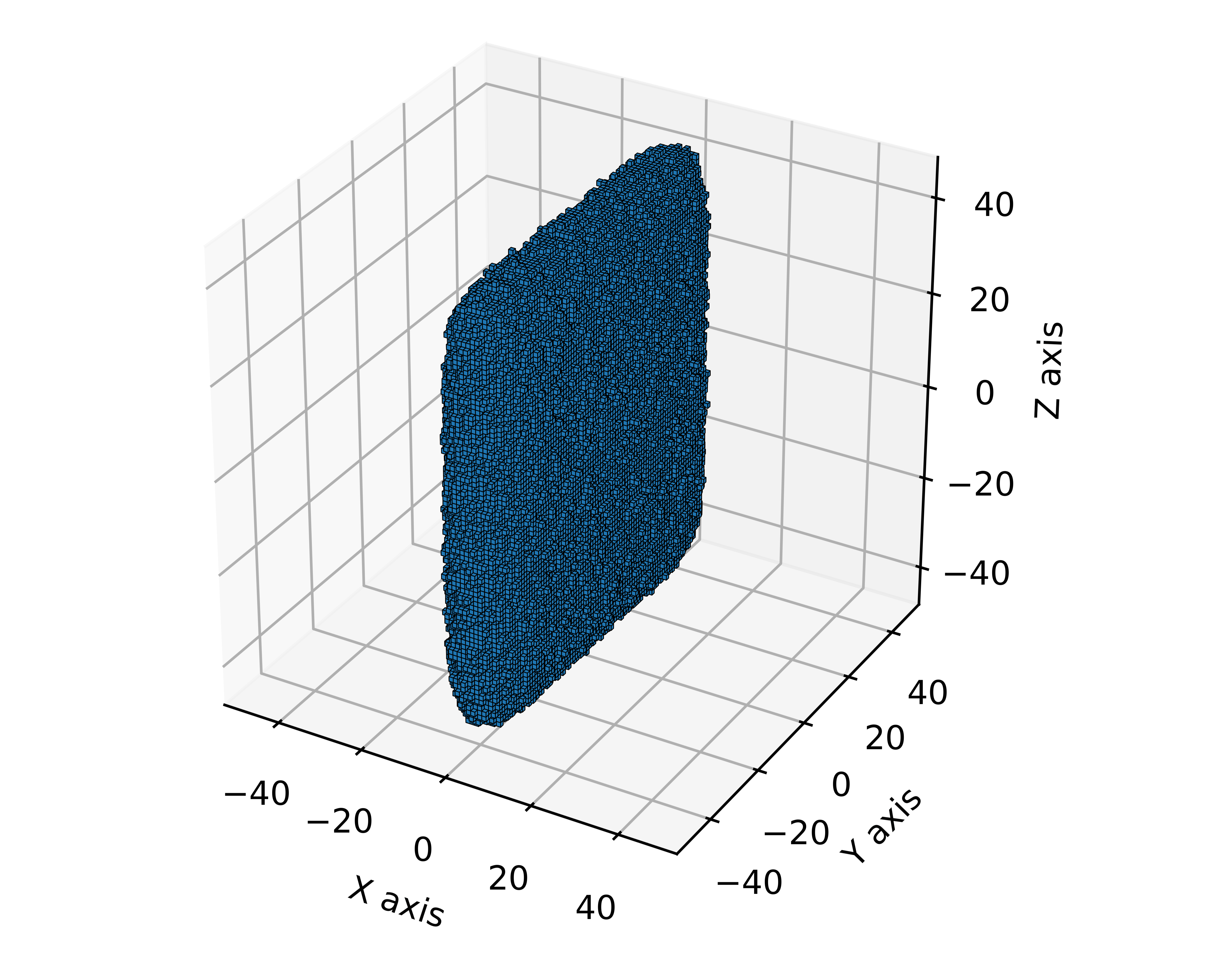}
	\caption{A realization of $A_{20}[40]$. Each particle is represented by a cube.}
	\label{fig: fullplot}
	\includegraphics{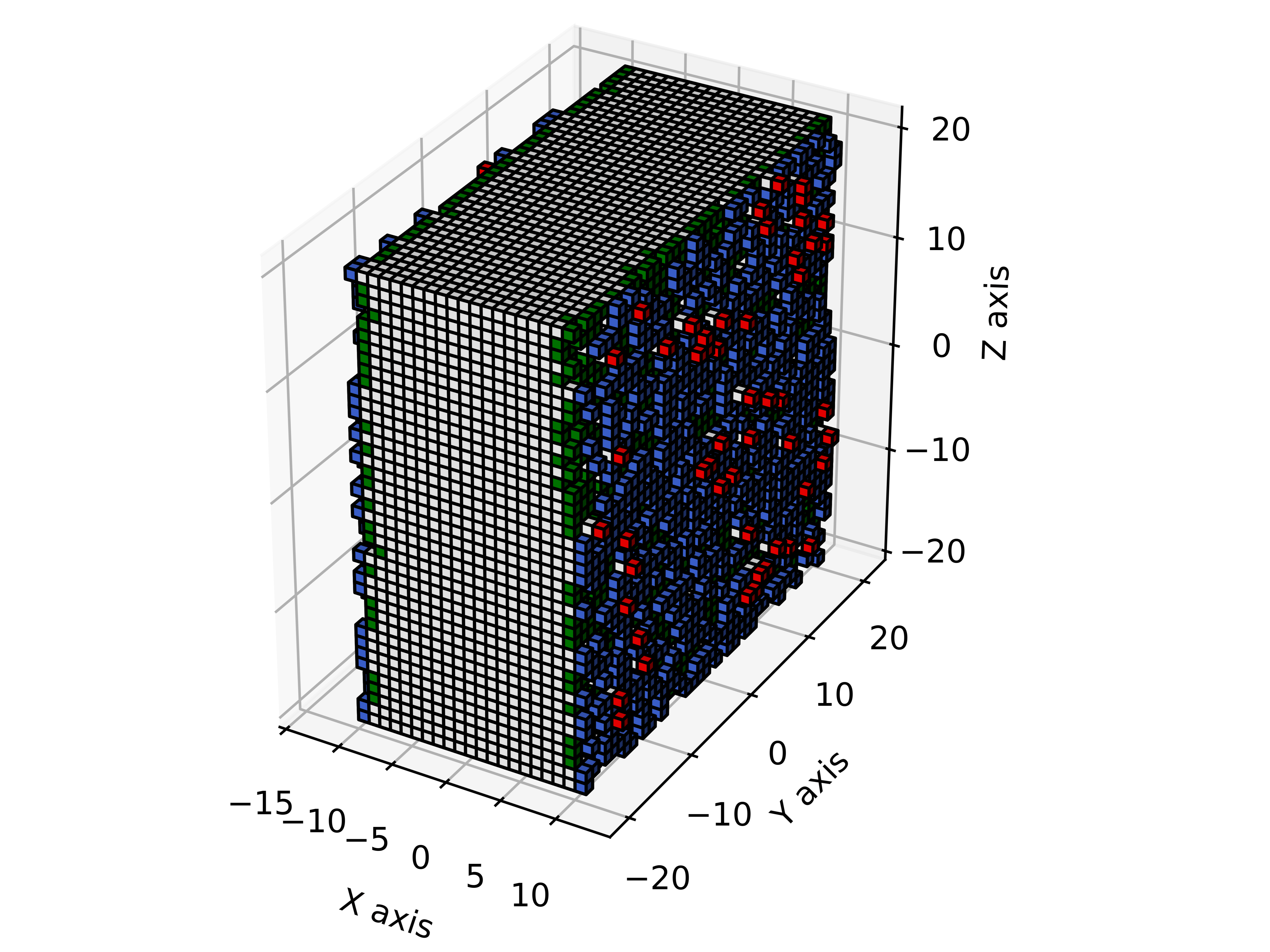}
	\caption{A realization of $A_{20}[40]\cap \mathbb{Z}_{20}$. The points with $x$-coordinate on the border such that $|x| = 10$ (resp. $|x|<10$ and $|x|>10$) are colored in blue (resp. green and red). All other points are colored in white.}
	\label{fig: shapethm}
\end{figure}
One of our main results is a shape theorem for $A_n[\infty]$. Restricted to the (large) strip $\mathbb{Z}_{n^\alpha} := \mathbb{Z} \times \llbracket -\lfloor n^\alpha\rfloor , \lfloor n^\alpha\rfloor \rrbracket^{d-1}$, the aggregate $A_n[\infty]$ looks like a slab with thickness $n$ and sublogarithmic fluctuations as the number of particles $n$ tends to infinity. Let us specify that the slab $\mathcal{R}_x$ is defined as \comment{$\mathcal{R}_x := \llbracket -\lfloor x \rfloor , \lfloor x \rfloor \rrbracket \times \mathbb{Z}^{d-1}$} for any positive real number $x$.
\begin{theorem}{(Shape theorem)}
	\label{thm: shape}
	For any integers $d\geq 3$ and $\alpha \geq 1$, there exists a constant $C = C(d,\alpha)>0$ such that, almost surely, there exists an integer $N\geq 1$ such that for any integer $n\geq N$,
	\begin{equation}
		\label{ShapeTh}
		\mathcal{R}_{n/2-C\sqrt{\log n}} \cap \mathbb{Z}_{n^\alpha} \subset A_n[\infty] \cap \mathbb{Z}_{n^\alpha} \subset \mathcal{R}_{n/2+C\sqrt{\log n}} \cap \mathbb{Z}_{n^\alpha} ~.
	\end{equation}
\end{theorem}

Let us comment on this shape theorem (see Figures \ref{fig: fullplot} and \ref{fig: shapethm}). It says that at first order, the limiting aggregate $A_n[\infty]$ is of thickness $n$, which makes sense since $n$ particles are launched per source. Notice that Proposition \ref{prop: avg nb of particles on a line} confirms that fact; $n$ is  (exactly) the mean thickness of $A_n[\infty]$. Let us also remark that Theorem \ref{thm: shape} holds for the aggregate $A_n[\infty]$ restricted to the strip $\mathbb{Z}_{n^\alpha}$ (even large). Such a restriction is unavoidable since a.s.\, there  exists some pathological source $z$ (far away from the origin) for which all the $n$ particles always move in the direction of the abscissa, meaning that the site $z+(n,0,\ldots,0)$ belongs to $A_n[\infty]$. Furthermore, Theorem \ref{thm: shape} specifies the fluctuations of the aggregate $A_n[\infty]$ around its limiting shape $\mathcal{R}_{n/2}$ (both restricted to the strip $\mathbb{Z}_{n^\alpha}$). They are (at most) sublogarithmic while they are (at most) logarithmic in dimension $d=2$ \cite{chenavier2023bi}. This dichotomy between dimension $d=2$ and higher echoes the results of \cite{asselah2013sublogarithmic,JLS13}, in which it is proved that the fluctuations for the standard IDLA are also sublogarithmic when $d\geq 3$.

\commentK{Our second main result is a stabilization result for the aggregate $\An$. It ensures that particles emitted from sources far away from the origin do not reach regions near the origin. It puts forward an independence property between the aggregate $A_n[\infty] \cap \mathbb{Z}_M$ and  particles from afar, which, in itself, is a very interesting property. While we will not be using it in this present paper, such a property will be crucial in \cite{chenavier2025construction}, where we use it to generalize the constructions of IDLA forests from \cite{chenavier2023bi} to higher dimensions. }
\begin{theorem}{(Strong stabilization)}
	\label{thm: strong stab}
	Let $n\geq 0$ and $\alpha > 1$ . A.s.\,there exists an integer $M_0$ such that, for any integer $M \geq M_0$, the trajectory of any particle contributing to $A_n[\infty]$ and starting from a level larger than $M^\alpha$ does not visit the strip $\mathbb{Z}_{M}$.
\end{theorem}
In what follows, we assume $\alpha \geq 2$ to be an integer, as picking real values of $\alpha$ requires a heavy use of floor functions. This choice is made simply for the sake of lightening notation. 

Theorem \ref{thm: strong stab} is an extension of Theorem 3.1 of \cite{chenavier2023bi} (concerning the bidimensional case) to dimension $d \geq 3$. As we explain now, this extension is non-trivial and its proof requires a new approach. As in \cite{chenavier2023bi}, particles contributing to the aggregate $A_n[\infty]$ are sent by successive waves, i.e.\,from the annuli
\[
\mathrm{Ann}(M,j) := \mathcal{H} \cap \big( \mathbb{B}((j+2)M^\alpha) \setminus \mathbb{B}((j+1)M^\alpha ) \big), \; j \geq 0,
\] 
where $\mathbb{B}(\ell)$ denotes the \comment{(closed)} ball with radius $\ell$ and centered at the origin (w.r.t.\,the supremum distance $\| \cdot \|$). 
When $d=2$, the hyperplane of sources $\mathcal{H}$ corresponds to the vertical axis and $\mathrm{Ann}(M,j)$ admits only $2M^{\alpha}$ sources, for any $j$. When $d \geq 3$, $\# \mathrm{Ann}(M,j)$ \comment{depends also on $j$} and increases with $j$ as $j^{d-2}$ (this factor disappears when $d=2$). The same holds for the number of particles sent during the $(j+1)$-th wave, i.e.\,from $\mathrm{Ann}(M,j)$. In order to visit the strip $\mathbb{Z}_M$ before stopping, a particle sent during the $(j+1)$-th wave has to travel inside the current aggregate until reaching $\mathbb{Z}_M$. It is more or less likely according to the index $j$ and the thickness of the current aggregate which can then be viewed as a 'random environment' where the particle evolves before stopping. However, there is a certain deterioration of the 'environment' when successive waves are launched. Indeed, if $A_{j+1}$ denotes the aggregate obtained after sending the $j$-th wave, then particles of the $(j+1)$-th wave contribute to the growth of $A_{j+1}$ into $A_{j+2}$ (i.e.\,$A_{j+2}$ is thicker than $A_j$) making easier the travel inside the current aggregate to the strip $\mathbb{Z}_M$ for further particles. Hence, we have to deal with two opposite trends: as $j$ increases, particles of the $(j+1)$-th wave have to travel a longer way to reach $\mathbb{Z}_M$, but this way is more likely since the corresponding aggregate is thicker. In dimension $d=2$, the number of particles sent at each wave being weak (and constant w.r.t.\,$j$), the deterioration phenomenon of the 'environment' is negligible compared to the distance that particles must travel and the stabilization result is not too difficult to obtain in this case (see Section 3.1 of \cite{chenavier2023bi}). In dimension $d \geq 3$, because of the increase of the number of particles sent at each wave, the deterioration of the 'environment' previously mentioned is stronger and the proof used in \cite{chenavier2023bi} no longer applies. To address this issue, the idea consists in proving that the aggregate $A_n[\infty]$, beyond some level, is included within a cone centered at the origin (Theorem \ref{prop: BM}). This upper bound presents two advantages. First it is rough enough-- the thickness of the cone increases when one moves away from the origin --to take into account the deterioration of the 'environment' phenomenon and pathological sources (previously cited). Second, it is global since it concerns the whole aggregate outside some compact set. This result is referred to as a \textit{rough global upper bound}.\\

Our paper is organized as follows. In Section \ref{section: first properties}, we give some properties of $A_n[\infty]$ including invariance (in distribution) w.r.t.\,translations/symmetries and a mass transport principle. We also recall the so-called \textit{Abelian property} which ensures that the order in which the particles are sent is not important (in distribution) to define $A_n[\infty]$. In Section \ref{section: donut method}, we discretize our problem into donuts and establish a result (Proposition \ref{prop: crossing prob}) which will be used to derive the rough global upper bound. This upper bound is stated and proved in Section \ref{section: global upper bound}. In the last two sections, we prove Theorems \ref{thm: shape} and \ref{thm: strong stab}. 

\section{First properties}
\label{section: first properties}

\subsection{Mass transport property and symmetries}
In this section, we state some basic properties satisfied by the random aggregate $\An$. The first property states that given a line $\mathbb{Z}\times\{\comment{y}\}$, where $\comment{y}\in\mathbb{Z}^{d-1}$, the average amount of particles that settle on this line is equal to $n$. Here, we have written $\comment{\{y\}:=\{(y_2,y_3,\ldots, y_d)\}}$ for any $\comment{y=(y_2,y_3,\ldots, y_d)}\in \mathbb{Z}^{d-1}$. One can interpret this as the following statement: on average, the $n$ particles sent from each source $(0, \comment{y})$ settle on the line $\mathbb{Z}\times\{\comment{y}\}$.
\begin{proposition}
	\label{prop: avg nb of particles on a line}
	Let $n\geq 1$. For all $\comment{y}\in \mathbb{Z}^{d-1}$
	\[
	\mathbb{E}\left[\#\left(A_n[\infty]\cap\left(\mathbb{Z}\times\{\comment{y}\}\right)\right)\right]=n.
	\]
\end{proposition}
Just as in Section 2 of \cite{chenavier2023bi}, a consequence of Proposition \ref{prop: avg nb of particles on a line} is a result of \emph{weak stabilization}, which claims that a particle sent far from the origin does not \emph{settle} close to the origin. This result differs from our result of strong stabilization given in Section \ref{section: strong stab}, as the latter shows that a particle sent far from the origin does not \emph{visit} areas close to the origin. Moreover, unlike strong stabilization, weak stabilization does not provide any exploitable bounds, which makes it impossible to use arguments such as the Borel-Cantelli Lemma. 

In the following proposition, we claim that the distribution of the random aggregate $A_n[\infty]$ is invariant with respect to translations and symmetries.
In what follows, we denote by $T_k$ the translation operator with respect to vector $k\in \mathcal{H}$ and $S_{k'}$ the point reflection operator across a point \comment{$k' \in \frac{1}{2}\mathcal{H}$}, that is:
\[
\forall x \in \mathbb{Z}^d,\ T_k(x):=x+k \quad \textrm{ and } \quad S_{k'}(x)=2k'-x.
\]
For $ B \subset \mathbb{Z}^{d-1}$, let
\[
T_kB=\{T_k(x),\ x \in B \} \quad \textrm{ and } \quad S_{k'}B=\{S_{k'}(x),\ x \in B \}.
\]
\begin{proposition}
	\label{prop: symmetry invariance}
	Let $n\geq 0,\ k\in\mathcal{H},\ k' \in \frac{1}{2}\mathcal{H}$.
	\begin{enumerate} 
		\item The distribution of $A_n[\infty]$ is invariant with respect to $T_k$, \comment{i.e.} $T_k \An \overset{\mathrm{law}}{=}\An$.
		
		\item The distribution of $A_n[\infty]$ is invariant with respect to $S_{k'/2}$, \comment{i.e.} \newline $S_{k'/2} \An \overset{\mathrm{law}}{=}\An$.
	\end{enumerate}
\end{proposition}

\subsection{Abelian property}
We give here the Abelian property, which states that altering the order in which particles are sent does not change the \emph{law} of the aggregate. We begin by defining the Diaconis-Fulton \emph{smash sum}: (see \cite{diaconis1991growth}).
For $A\subset \mathbb{Z}^d$ \comment{(possibly random)} and $z\in \mathbb{Z}^d$:
\begin{itemize}
	\item if $z\not\in A$, then $A \oplus \{z\} = A\cup\{z\}$;
	\item if $z\in A$, then $A \oplus \{z\}$ is the random set obtained by adding to $A$ the vertex on which a simple random walk started in $z$, \comment{independent of $A$}, exits $A$. 
\end{itemize}  
\begin{proposition}[Abelian property]
	Let $A$ and $\{z_1,\ldots, z_k\}$ be subsets of $\mathbb{Z}^d$. The distribution of 
	\[
	((A\oplus \{z_1\}) \oplus \{z_2\})\oplus \cdots \oplus \{z_k\}
	\]
	does not depend on the order of the $z_i$'s. That is, if we take $\sigma \in \mathfrak{S}_k$ a permutation of $\{1,\ldots, k\}$, then:
	\[
	((A\oplus \{z_1\}) \oplus \{z_2\})\oplus \cdots \oplus \{z_k\} \overset{\mathrm{law}}{=}  ((A\oplus \{z_{\sigma(1)}\}) \oplus \{z_{\sigma(2)}\})\oplus \cdots \oplus \{z_{\sigma(k)}\}.
	\]
\end{proposition}

\subsection{Proofs of Propositions \ref{prop: avg nb of particles on a line} and \ref{prop: symmetry invariance}}
We only show the proof of Proposition \ref{prop: avg nb of particles on a line} since Proposition \ref{prop: symmetry invariance} can be dealt with in a similar manner. 
Our main idea is to build an auxiliary aggregate $A'_n[\infty]$ with the same law as $\An$, but for which it is simpler to show translation invariance. To do so, we construct $A'_n[\infty]$ in the same spirit as $\An$. Let $M\geq0$. We define $A'_n[M]$ similarly to $A_n[M]$, by sending $n$ particles per source of \comment{$\mathcal{H}_M:=\{0\}\times \llbracket -M, M \rrbracket^{d-1}$}, but this time the order is given by random clocks. More precisely, let $\left(\mathcal{U}_{z,i}\right)_{z\in \mathcal{H},\ 1\leq i \leq n }$ be a family of \comment{i.i.d.} uniform random variables on $[0,1]$. For each $z\in \mathcal{H}$ we can order these $n$ random variables in order to get an increasing family of clocks $(\tau_{z,i})_{1\leq i\leq n}$ in $[0,1]$. 
Now, with the collection of random clocks $\{\tau_{z,i}:\ z\in\mathcal{H},\ 1\leq i\leq n\}$ we can associate a family of independent symmetric random walks $\{S_{z,i}:\ z\in\mathcal{H},\ 1\leq i\leq n\}$ on $\mathbb{Z}^d$, \comment{issued in $z$} and independent of the family of clocks. Just like above, at time $\tau_{z,i}$, the $i$-th particle is sent from source $z\in \mathcal{H}$ and follows a trajectory given by $S_{z,i}$, adding a new site to the current aggregate. Let us specify that each particle's  trajectory is instantly realized and that it settles immediately. 
The aggregate $A'_n[M]$ is obtained following the same protocol as above by sending particles up to level $M$ according to the random clocks given by our family $\left(\mathcal{U}_{z,i}\right)_{\|z\| \leq M,\ 1\leq i \leq n}$. Using the Abelian property, we have 
\[
A'_n[M]\overset{\mathrm{law}}{=}A_n[M].
\]
By adapting Lemma 2.1 of \cite{chenavier2023bi}, we can easily show that a.s.\,for all $n, M \geq 0,$ \newline $A'_n[M]\subset A'_n[M+1]$. Then we define $A'_n[\infty]$ as the increasing union:
\[
A'_n[\infty]:=\bigcup_{M\geq 0}A'_n[M] \quad \text{ a.s. }
\]
Since both sequences $\left(A'_n[M]\right)_{M\geq 0}$ and $\left(A_n[M]\right)_{M\geq 0}$ are almost surely increasing and that $A'_n[M]\overset{\mathrm{law}}{=}A_n[M]$ for all $M\geq 0$, we have $A'_n[\infty]\overset{\mathrm{law}}{=}A_n[\infty]$.

We are now prepared to prove Proposition \ref{prop: avg nb of particles on a line}. Indeed, since $A'_n[\infty]\overset{\mathrm{law}}{=}A_n[\infty]$, it is sufficient to prove the same type of result for $A'_n[\infty]$. For $x,y \in \mathbb{Z}^{d-1}$, we let $Q_{x\to y}$ denote the number of particles sent from $(0,x)$ that settle on the line $\mathbb{Z}\times \{y\}$ \comment{in the construction of $A'_n[\infty]$}.
Now, for $A,B\subset \mathbb{Z}^{d-1}$, we define: 
\[
Q(A,B):=\mathbb{E}\left[\sum_{x\in A,\ y \in B} Q_{x\to y}\right].
\]
In particular, for all $\comment{y}\in\mathbb{Z}^{d-1}$, we have 
\[
\mathbb{E}\left[\#\left(\comment{A'_n[\infty]} \cap \left(\mathbb{Z}\times\{\comment{y}\}\right)\right)\right] = Q(\mathbb{Z}^{d-1},\{\comment{y}\}).
\]
Since $Q(\{\comment{y}\},\mathbb{Z}^{d-1})=n$, it is sufficient to prove that $Q(\mathbb{Z}^{d-1},\{\comment{y}\})=Q(\{\comment{y}\},\mathbb{Z}^{d-1})$. We show this using a mass transport argument (see Theorem 5.2 of \cite{benjamini2001percolation}). It is sufficient to show that $Q$ is diagonally invariant, that is: $Q(A+w,B+w)=Q(A,B)$ for all $w\in \mathbb{Z}^{d-1}$. This holds, since
\begin{align*}
	Q(A+w,B+w)&=\sum_{x\in A,\ y \in B}\mathbb{E}\left[Q_{x+w\to y+w}\right]\\
	&=\sum_{x\in A,\ y \in B}\mathbb{E}\left[Q_{x\to y}\right]\\
	&=Q(A,B),
\end{align*}
where the second line comes from the fact that 
\begin{align}
	Q_{x+w\to y+w}&=Q_{x+w\to y+w}\left((\tau_{z,i})_{\substack{z\in\mathcal{H}\\ 1\leq i\leq n}},\ (S_{z,i})_{\substack{z\in\mathcal{H}\\ 1\leq i\leq n}}\right)\label{eqn: Qxy} \\ 
	&\overset{\mathrm{a.s}}{=}Q_{x\to y}\left((\tau_{z-w,i})_{\substack{z\in\mathcal{H}\\ 1\leq i\leq n}},\ (S_{z-w,i})_{\substack{z\in\mathcal{H}\\ 1\leq i\leq n}}\right) \notag \\ 
	&\overset{\mathrm{law}}{=}Q_{x\to y}. \notag
\end{align}
Note that the computations in \eqref{eqn: Qxy} are specific to $A'_n[\infty]$, and are \emph{not} true for $\An$. The key argument here is that the particles in $A'_n[\infty]$ are not sent according to a specific order but according to a family of independent uniform clocks, implying that all particles play the same role for the aggregate.

\section{The donut method}
\label{section: donut method}

In this section, we introduce what will be commonly referred to throughout this paper as the \emph{donut method}. This argument will be particularly useful when coupled with the global upper bound given in Section \ref{section: global upper bound} to control the trajectory of a given particle. The method consists in building donuts, starting from the origin and up to any given level, and showing that a particle is unlikely to cross multiple donuts without settling beforehand. 
Let us begin by detailing the construction of our donuts, for which it is necessary to first define \emph{cones}. For $\varepsilon>0$, we define the cone of angle $\varepsilon$ as:
\begin{equation}
	\label{eqn: cone}
	\mathscr{C}_{\varepsilon}:=\bigcup_{l\geq 0}\Bigl\{z\in \mathbb{Z}^d,\ \|p_{\mathcal{H}}(z)\|=l,\ |z_1|\leq \varepsilon l	\Bigr\},
\end{equation}
where $p_{\mathcal{H}}$ is the operator realizing the orthogonal projection on $\mathcal{H}$. 
Let $\mathbb{B}_{d-1}(r)$ denote the $(d-1)$-dimensional lattice ball of radius $r$, that is 
\[
\forall r>0,\ \mathbb{B}_{d-1}(r):=\{x \in \mathbb{Z}^{d-1}: \|x\|\leq r\}.
\] 
Given a decreasing family of real numbers $(l_i)_{i\geq 0}$, we define the donut $\textbf{D}^i$ as:
\[
\forall i\geq 0,\ \textbf{D}^i:=\comment{\llbracket -\varepsilon l_i, \varepsilon l_i \rrbracket} \times \left(\mathbb{B}_{d-1}(l_i)\setminus \mathbb{B}_{d-1}(l_{i+1}) \right)\comment{.}
\]
We build each donut $\textbf{D}^i$ so that its length, which equals $2\varepsilon l_i$, is equal to its width $l_i-l_{i+1}$ (see Figure \ref{fig: donut example}: one may see this figure as the view along a vertical cut of our donuts in dimension 3).
This gives the following condition on $(l_i)_{i\geq 0}$:
\[
\forall i\geq 0,\ l_{i+1}=(1-2\varepsilon)l_i,
\]
with $\varepsilon< 1/2$. By induction, we get the general expression:
\[
\forall i\geq 0,\ l_i=(1-2\varepsilon)^{i}l,
\]
where $l=l_0$.
We consider the number of donuts between levels $l$ and $M$, with $M<l$, and define $k = k(l, M, \varepsilon)$ as the greatest integer such that:
\[
\sum_{i=0}^k 2\varepsilon l_i\leq l-M.
\]
Since $l_i=(1-2\varepsilon)^il$, for $\varepsilon$ taken small enough, we have:
\begin{equation}
	\label{eqn: value of k}
	k\geq \underbrace{\dfrac{-1}{2\log(1-2\varepsilon)}}_{K(\varepsilon)}\times \log\left(\dfrac{l}{M}\right).
\end{equation}
Notice here that $K(\varepsilon)$ can be taken arbitrarily large by taking $\varepsilon$ arbitrarily small. 
\begin{figure}[!h]
	\centering
	\begin{tikzpicture}[scale=0.3]
		%\draw (-6,-2) grid (6,20);
		\fill[color=gray!40] (-2.125,21.25) -- (-0.6,6) -- (0.6,6) -- (2.125,21.25) -- cycle;
		\draw (0, 19.5) -- (0.5, 19.5) ;
		\draw (0.5, 20) -- (0.5, 19.5);
		\filldraw (0,0) circle (1pt);
		\draw[->,very thick] (-6,0) -- (6,0);
		\draw[->,very thick] (0,-2) -- (0,21.5);
		\draw (0,0) -- (-2.125,21.25);
		\draw (0,0) -- (2.125,21.25);
		\draw [<->] (2.5,20) -- (2.5,16);
		\draw [<->] (2.1,16) -- (2.1,12.8);
		\draw [<->] (1.78,12.8) -- (1.78,10.24);
		%		\draw (1,20.25) node[above] {$l\varepsilon$};
		\draw (2.5,18) node[right] {$2\varepsilon l$};
		\draw (2.1,14.4) node[right] {$2\varepsilon l_1$};
		\draw (1.78,11.52) node[right] {$2\varepsilon l_2$};
		\draw (0,20.5) node[left] {$l$};
		\draw (0,16.5) node[left] {$l_1$};
		\draw (0,13.3) node[left] {$l_2$};
		\draw (0,10.74) node[left] {$l_3$};
		\draw (0,8.692) node[left] {$l_4$};
		\draw (0,7.0536) node[left] {$l_5$};
		\draw (0,20) node {$\times$};
		\draw (0,16) node {$\times$};
		\draw (0,12.8) node {$\times$};
		\draw (0,10.24) node {$\times$};
		\draw (0,8.192) node {$\times$};
		\draw (0,6.5536) node {$\times$};
		\draw [teal](-2,20) rectangle (2,16);
		\draw [teal](-1.6,16) rectangle (1.6,12.8);
		\draw [teal](-1.28,12.8) rectangle (1.28,10.24);
		\draw [teal](-1.024,10.24) rectangle (1.024,8.192);
		\draw [teal](-0.8192,8.192) rectangle (0.8192,6.5536);
		%		\draw [dashed, red] (-7,9)--(7,9);
		\draw [dashed, blue] (-7,6)--(7,6);
		%		\draw [red] (-7,9) node[left] {$M^{\alpha}$};
		\draw [blue] (-7,6) node[left] {$M$};
		\draw (8,0) node[below] {$\mathbb{Z}\times\{0\}$};
	\end{tikzpicture}
	\caption{Partition into donuts}
	\label{fig: donut example}
\end{figure}

Let us now briefly explain the reasoning behind the construction of our donuts. Our method will be particularly useful to show that a particle sent far away from the origin is highly unlikely to travel close to the origin while staying \emph{within} the cone. For a particle to do so it necessarily has to travel through many donuts without ever exiting the cone, since the donuts are built in such a way that they wrap around the cone $	\mathscr{C}_{\varepsilon}$. Such an event is handled by the following Proposition.
\begin{proposition}
	\label{prop: crossing prob}
	Let $M\geq 1$ and $\varepsilon>0$. Fix $(S_t)_{t\geq 0}$ a simple symmetric random walk starting from some source of $\mathcal{H}\setminus\mathcal{H}_{M}$ and consider the cone $\mathscr{C}_{\varepsilon}$ defined as in \eqref{eqn: cone}. For $i\geq 1$, let
	\[
	A_i=\left\lbrace\begin{array}{c}
		\mbox{ The walk crosses the $i$ donuts $\textbf{D}^0,\dots,\ \textbf{D}^{i-1}$ }\\
		\mbox{ without exiting the cone $	\mathscr{C}_{\varepsilon}$}\\
	\end{array}
	\right\rbrace ,
	\]
	and let $A_0=\Omega$.
	Then, for any $i\geq 0$,
	\[
	\mathbb{P}\left(A_i\right)\leq (1-c)^i,
	\]
	where $c=(2d)^{-2}$.
\end{proposition}
Note that what we mean by a walk or particle crossing donut $\textbf{D}^i$ is for it to reach the inner ring of $\textbf{D}^{i}$ without ever exiting $\textbf{D}^{i}$.
\newline

Notice that for a walk to cross a donut (from the outer ring to the inner ring), it already needs to get through the middle of that donut.  To deal with this property, let us introduce the notion of 'middling slice' of a donut. Let $i\geq 0$ and consider the $i$-th donut ${\textbf{D}^i =\llbracket -\varepsilon l_i, \varepsilon l_i \rrbracket \times \left(\mathbb{B}_{d-1}(l_i)\setminus \mathbb{B}_{d-1}(l_{i+1}) \right)}$. Define the lattice sphere of radius $s$ as:
\[
\forall s\geq 0,\ \mathbb{S}_{d-1}(s):=\{x\in\mathbb{Z}^{d-1},\ \|x\|=s\}.
\]
Now, notice that $\frac{l_i +l_{i+1}}{2}=(1-\varepsilon)l_i$. Define the middling slice of $\textbf{D}^i$ as 
\[
\Cmid:=\comment{\llbracket -\varepsilon l_i, \varepsilon l_i \rrbracket} \times \mathbb{S}_{d-1}\left((1-\varepsilon)l_i \right).
\]
Additionally, define the exterior border of $\textbf{D}^i$ as 
\[
\Dext^i=\Bigl(\comment{\rrbracket -\infty,-\varepsilon l_i\rrbracket \cup \llbracket \varepsilon l_i,+\infty\llbracket }\Bigr) \times \Bigl(\mathbb{B}_{d-1}(l_i)\setminus \mathbb{B}_{d-1}(l_{i+1})\Bigr).
\]
The following result shows that a walk started from the middling slice of a donut has a positive probability of exiting the donut through $\Dext^i$ and will be used to derive Proposition \ref{prop: crossing prob}. 
\begin{lemma}
	\label{lemma: Dext prob}
	Let $y \in \emph{\textbf{m}}_i$  and let $(S_t)_{t\geq 0}$ be a simple symmetric random walk on $\mathbb{Z}^d$ started at $y$. For all $i\geq 0$, we introduce the stopping time $\tau_y=\inf\{t\geq 0,\ S_t \notin \mathbb{B}_d(y,\varepsilon l_i)\}$. We have:
	\[
	\mathbb{P}_y\left(S_{\tau_y}\in \emph{\textbf{D}}_{\mathrm{ext}}^i\right)\geq \dfrac{1}{2d}.
	\]
\end{lemma}
\begin{proof}[Proof of Lemma \ref{lemma: Dext prob}:]
	Let $y\in \Cmid$. Notice that $\mathbb{B}_d(y,\varepsilon l_i) \subseteq \mathbb{Z}\times \Bigl(\mathbb{B}_{d-1}(l_i)\setminus \mathbb{B}_{d-1}(l_{i+1})\Bigr)$, and that $\mathbb{B}_d(y,\varepsilon l_i)$ has at least one of its $2d$ faces, say $\mathbf{F}$, included in $\Dext^i$.
	By an argument of symmetry, we have 
	\[
	\mathbb{P}_y(S_{\tau_y} \in \mathbf{F})=\dfrac{1}{2d}.
	\]
	Now, since $\mathbb{P}_y(S_{\tau_y} \in \mathbf{F}) \leq \mathbb{P}_y(S_{\tau_y}\in \Dext^i)$, we have the desired result.
\end{proof}

\begin{proof}[Proof of Proposition \ref{prop: crossing prob}]
	The case where $i=0$ is trivial. Let ${i\geq 1}$. Notice that the sequence of events $(A_i)_{i\geq 0}$ is decreasing, so 
	$
	\mathbb{P}(A_i)=\mathbb{P}(A_i | A_{i-1})\mathbb{P}(A_{i-1}).
	$
	Thus, it is sufficient to prove that ${\mathbb{P}(A_i | A_{i-1})\leq (1-c)}$.
	Since we are considering events where the walk crosses donuts from outer ring to inner ring, we will refer to \emph{good sides} as sides orthogonal to the '$x$' axis, whereas \emph{bad sides} will refer to sides that are \emph{not good}.
	
	Let us define the following events:
	\[
	M_i = \biggl\{ \mbox{The random walk reaches $\textbf{m}_i$ } \biggl\},\\
	\]
	\[
	D_i = \left\lbrace\begin{array}{c}
		\mbox{  The random walk exits the $i$-th donut $\textbf{D}^{i-1} $ }\\
		\mbox{ through \comment{one of the bad sides} }\\
	\end{array}
	\right\rbrace .
	\]
	Additionally, define the sequence of stopping times $T_i:=\inf\{t\geq 0,\ S_t \in \textbf{m}_i\}$. As mentioned earlier, for the walk to cross a donut (from outer ring to inner ring) it necessarily has to cross the middling slice of the donut, and since $\mathscr{C}_{\varepsilon}\cap \mathbb{Z}_M^c \subset \bigcup_{j\geq 0} \textbf{D}^{j}$, this implies that on the event $A_i$, the walk crossed the $i$ donuts $\textbf{D}^0,\dots,\ \textbf{D}^{i-1}$ without ever exiting through a good side. Therefore:
	\begin{align*}
		\mathbb{P}(A_i | A_{i-1})&\leq \mathbb{P}(M_i\cap D_i | A_{i-1})\\
		&\leq \sum_{m\in \textbf{m}_i}\mathbb{E}\left[\mathbbm{1}_{D_i\cap M_i}\mathbbm{1}_{S_{T_i}=m} |\ A_{i-1}\right]\\
		&\leq \sum_{m\in \textbf{m}_i}\mathbb{E}\left[\mathbbm{1}_{D_i}\mathbbm{1}_{S_{T_i}=m} |\ A_{i-1}\right]\\
		&\leq \sum_{m\in \textbf{m}_i}\mathbb{P}\left(D_i |\ S_{T_i}=m,\ A_{i-1}\right)\mathbb{P}(S_{T_i}=m |\ A_{i-1}).
	\end{align*}
	Now, by the Markov property, for all $m\in \Cmid,\ \mathbb{P}\left(D_i |\ S_{T_i}=m,\ A_{i-1}\right)\leq\mathbb{P}_m(D_i)$.
	It remains to bound $\mathbb{P}_m(D_i)$, 
	which is an immediate consequence of Lemma \ref{lemma: Dext prob}. Let us first define 
	$
	\partial\Dext^i:=\{-\varepsilon l_i-1,\varepsilon l_i+1\}\times \Bigl(\mathbb{B}_{d-1}(l_i)\setminus \mathbb{B}_{d-1}(l_{i+1})\Bigr).
	$
	%These are the points right outside of $\textbf{D}^i$, on the good sides.  
	Notice that if the walk hits a site of $\partial\Dext^i$, then it has necessarily exited the donut through a good side. Hence, using the result of Lemma \ref{lemma: Dext prob}:
	\begin{align*}
		\mathbb{P}_m(D_i^c)&\geq\mathbb{P}_m\left(\{S_{\tau_m}\in \Dext^i\} \cap \{S_{\tau_m +1} \in \partial\Dext^i\}\right)\\
		&\geq \mathbb{P}_m(S_{\tau_m +1} \in \partial\Dext^i | S_{\tau_m}\in \Dext^i )\mathbb{P}_m(S_{\tau_m}\in \Dext^i)\\
		&\geq \mathbb{P}_{\Dext^i}(S_1 \in \partial\Dext^i)\mathbb{P}_m(S_{\tau_m}\in \Dext^i)\\
		&\geq\dfrac{1}{2d}\times \dfrac{1}{2d}.
	\end{align*}
	This concludes the proof, since 
	\begin{align*}
		\mathbb{P}(A_i | A_{i-1}) &\leq \sum_{m\in \textbf{m}_i} \mathbb{P}_m(\comment{D_i})\mathbb{P}(S_{T_i}=m |\ A_{i-1})	\\
		&\leq \left(1-\frac{1}{4d^2}\right)\sum_{m\in \textbf{m}_i} \mathbb{P}(S_{T_i}=m |\ A_{i-1}) \leq 1-\frac{1}{4d^2}.
	\end{align*} 
\end{proof}

\section{A rough global upper bound}
\label{section: global upper bound}

As seen in dimension 2 (see \cite{chenavier2023bi}, Section 6), when restricted to a certain level, the aggregate $A_n[\infty]$ is contained within a rectangle of length \comment{roughly} $n$, with high probability. In this section, we prove that above a certain level the aggregate is entirely contained within a cone with high probability. To state the result, we first give some notation. For any source $z \in \mathcal{H}$ and given a realization of $A_n[\infty]$, we define:
\[
X_z(n):=\max\Bigl\{|z'_1|,\ z'\in A_n[\infty],\ z'_i=z_i \quad \forall i=2,\dots,d \Bigr\}.
\]
The random variable $X_z(n)$ is the absolute value of the first coordinate of the furthest occupied site on the line of level $z$. 
Moreover, for any $0<\varepsilon$ and $M\geq 1$ we let:
\[
\B{n}=\bigcup_{l\geq M}\{\exists z \in \mathcal{H}: \|z\|=l,\ X_z(n)>\varepsilon l\}.
\]
The event $\B{n}$ describes the situation where one or more particles have settled at a distance greater than $\varepsilon l$ on some line of distance $l\geq M$ from the origin. The following proposition shows that such an event occurs with small probability.

\begin{theorem}
	\label{prop: BM}
	Let $n\geq 1$. \comment{For all $L>1$, for all $\varepsilon>0$, there exists a positive constant $C_{\varepsilon,n}$ such that for all $M\geq 2$} :	
	\[
	\mathbb{P}\left(\B{n}\right)\leq \dfrac{C_{\varepsilon,n}}{M^L} ~.
	\]
\end{theorem}

As a consequence of the above result, a.s.\,there exists a random integer $M_0$ such that for any $M \geq M_0$, the aggregate $\An \cap \mathbb{Z}_M^c$ is included in $\mathscr{C}_{\varepsilon}$. Theorem \ref{prop: BM} can be understood as the fact that $\An\cap \mathbb{Z}_M^c$ is included within $\mathscr{C}_{\varepsilon}$ with high probability since 
\[
\Bbis{n}=\{\An \cap \mathbb{Z}_M^c \subset \mathscr{C}_{\varepsilon} \}.
\] The property $\An \cap \mathbb{Z}_M^c \subset\mathscr{C}_{\varepsilon}$ is referred to as the \textit{rough global upper bound}.  This upper bound  will be very useful when coupled with the donut argument of Section \ref{section: donut method}, as it will allow us to show that particles are unlikely to travel long distances while staying within the cone.

The proof of Theorem \ref{prop: BM} will be shown by induction over $n$. Our idea is that if for some fixed $n$, the aggregate $\An$ is contained within a cone of angle $\varepsilon$ for some $\varepsilon>0$, then we show that after launching an additional particle from each source, the resulting aggregate is very likely contained in a slightly larger cone of angle $\varepsilon'>\varepsilon$.
To prove Theorem \ref{prop: BM}, we first show a stronger version in Proposition \ref{prop: BMN}. Before we  give this result, we need to build two increasing sequences $(M_n)_{n\geq 0}$ and $(\varepsilon_n)_{n\geq 0}$, corresponding to levels of particles and successive angles of cones.
Let $\varepsilon \in \interoo{0 1},\ M\geq 1$. We define the sequences $(M_n)_{n\geq 0}$ and $(\varepsilon_n)_{n\geq 0}$ by induction:
\[
\left\{
\begin{array}{ll}
	M_0=M\\
	\forall n\geq 0,\ M_{n+1}=M_n\left(1-\dfrac{\varepsilon}{2^{n+1}}\right)^{-1}
\end{array}
\right.\quad \quad
\left\{
\begin{array}{ll}
	\varepsilon_0=\varepsilon \\
	\forall n\geq 0,\ \varepsilon_{n+1}=\varepsilon_n+\dfrac{\varepsilon}{2^n}
\end{array}
\right.
\]
Note that for all $n\geq 0$, we have $\varepsilon\leq \varepsilon_n\leq 2\varepsilon$ and $M\leq M_n <2M$, \comment{for $\varepsilon$ small enough}.
\newline
We now give a stronger version of Theorem \ref{prop: BM}. To avoid any heavy notation, in what follows, we will write $\Bn{n}$ rather than $\mathrm{\textbf{Over}}(M_n,\varepsilon_n,n)$. 
Additionally, we continue to omit writing the floor function $\lfloor\cdot\rfloor$. 
\begin{proposition}
	\label{prop: BMN}
	For all $L>1$, for all $\varepsilon>0$, for all $n\geq 0$, there exists a constant $C_{\varepsilon,n}>0$ such that for all $M\geq 1$,
	\[
	\mathbb{P}\left(\Bn{n}\right)\leq \dfrac{C_{\varepsilon,n}}{M^L}.
	\]
\end{proposition}

\begin{proof}[Proof of Theorem \ref{prop: BM}:]
	Let $n$ be fixed. Using the fact that for all $M\geq 1$ and for all $\varepsilon$ small enough, $\varepsilon\leq \varepsilon_n\leq 2\varepsilon$ and $M\leq M_n <2M$, we have $\Over(2M,2\varepsilon,n) \subset \Bn{n}$, hence
	\[
	\mathbb{P}\left(\Over(2M,2\varepsilon,n)\right)\leq \mathbb{P}\left(\Bn{n}\right)\leq \dfrac{C_{\varepsilon,n}}{M^L}.
	\]
\end{proof}

\begin{proof}[Proof of Proposition \ref{prop: BMN}:]
	We prove our result by induction over $n$. Take $L>1$. Our induction statement is the following:
	\[
	\comment{\forall n\geq 0,\ \mathcal{P}(n): \forall \varepsilon\in \interoo{0 1},\ \exists C_{\varepsilon,n}>0, \forall M\geq 1,\ \mathbb{P}\left(\Bn{n}\right)\leq \dfrac{C_{\varepsilon,n}}{M^L}.}
	\]
	When $n = 0$, we have that $A_n[\infty] \overset{\mathrm{a.s}}{=} \emptyset$ and hence $A_n[\infty] \cap \mathbb{Z}_M^c \overset{\mathrm{a.s}}{\subset}\mathscr{C}_{\varepsilon}$, therefore $\mathbb{P}\left(\Bn{n}\right) = 0$.
	Let $n\geq 0$ and suppose $\mathcal{P}(n)$ holds. We write:
	\[
	\mathbb{P}\left(\Bn{n+1}\right)\leq \mathbb{P}\left(\Bn{n+1}\cap \Bbisn{n}\right)+\mathbb{P}\left(\Bn{n}\right).
	\]
	The right-hand term is handled by our induction hypothesis. We now focus on the left-hand term. 
	On the event $\Bn{n+1}\cap \Bbisn{n}$, we have $A_n[\infty] \cap \mathbb{Z}_{M_n}^c \subset \mathscr{C}_{\varepsilon_{n}} $, but when launching one additional particle from each source of $\mathcal{H}$, the new aggregate obtained spills over $\mathscr{C}_{\varepsilon_{n+1}}$ on $\mathbb{Z}_{M_{n+1}}^c$.
	This implies the existence of three random sites $(Z,Z^*,Z_{n+1})\in \mathbb{Z}^d$ such that:
	\begin{itemize}
		\item $Z^*$ is the source from which the first overflowing particle is emitted
		\item $Z_{n+1}$ is the site on which this particle settles
		\item $Z$ is the orthogonal projection of $Z_{n+1}$ on $\mathcal{H}$.
	\end{itemize} 
	Note that the coordinates of $Z_{n+1}$ are given by
	\[
	Z_{n+1}=Z\pm \left(\varepsilon_{n+1}\|Z\|\right)\cdot e_1,
	\]
	where $e_1=(1,0,\ldots, 0)$, and that these coordinates only depend on $Z$ \comment{(up to the sign)}, meaning it suffices to know the location of $Z$ to know precisely where the overflowing particle settled.
	\newline
	Additionally, we call $A_{Z^*}$ the aggregate (restricted to $\mathbb{Z}_{M_{n+1}}^c$) made up of $A_n[\infty]$ and each additional particle sent from $\mathcal{H}$ in the usual order up to site $Z^*$, and $A_{Z^*}^{-}$ the aggregate (restricted to $\mathbb{Z}_{M_{n+1}}^c$) made up of 
	$A_n[\infty]$ and each additional particle sent from $\mathcal{H}$ in the usual order up to site $Z^*$ \emph{excluded}. We know that this aggregate is contained inside of $\mathscr{C}_{\varepsilon_{n+1}}\cap \mathbb{Z}_{M_{n+1}}^c$. 
	Notice that  $A_{Z^*}=A_{Z^*}^{-}\cup\{Z_{n+1}\}$. \newline
	We write:
	\begin{multline}
		\label{eqn: doublesum}
		\mathbb{P}\left(\Bn{n+1}\cap\Bbisn{n}\right) \\
		\leq \sum_{l\geq M_{n+1}}\sum_{\|z\|=l}\mathbb{P}\left(Z=z,\ \Bn{n+1}\cap\Bbisn{n} \right).
	\end{multline}
	Now, fix $l\geq M_{n+1}$ and $z\in \mathcal{H}$ such that $\|z\|=l$, and let $z_{n+1}=z\pm \left(\varepsilon_{n+1}\|z\|\right)\cdot e_1$. To deal with the probability in \eqref{eqn: doublesum}, we consider two cases, which we show are both unlikely. The first case is the case where \comment{many particles have settled in a ball around $z_{n+1}$}, and the second is the case where a thin 'tentacle' has branched out towards $z_{n+1}$. The following lemma is an adaptation of Lemma 2 of \cite{jerison2012logarithmic}, which deals with the case of tentacles:
	\begin{lemma}
		\label{lemma: tentacle} 
		There exist positive universal constants $b,\ K_0,\ c$ such that for all real numbers $r>0$ and all $z\in\mathcal{H}$ with $0\notin \mathring{\mathbb{B}}(z_{n+1},r)$,
		\begin{multline*}
			\mathbb{P}\left(Z=z,\ \Bn{n+1}\cap\Bbisn{n},\ \#\left(A_{Z^*}\cap \mathbb{B}(z_{n+1},r)\right)\leq br^d\right)\\
			\leq K_0e^{-cr^2}.
		\end{multline*}
	\end{lemma}
	\begin{figure}[!h]
		\centering
		\begin{subfigure}{.45\textwidth}
			\centering
			\begin{tikzpicture}[use Hobby shortcut,closed=true]
				%\draw (-6,-6) grid (6,6);
				\fill [color=gray!40] (3.88,5.9).. (3,8).. (-1.5,7.5).. (-1,6.5).. (-2,2).. (-1.7,1).. (-2.2,-1).. (1.4,-1).. (2,1.7).. (3,5).. (3.88,5.9) ;
				\draw (0,-2) -- (0,8);
				%\draw (0,6) node {--};
				\draw (0,6) node[left] {$\|Z\|$};
				\draw (0.75,-2) --(2,8);
				\draw (1.375,-2) -- (4.5,8);
				%\draw (3.88,6) node {+};
				\filldraw (3.88,6) circle (1pt);
				\draw (4,6) node[right] {$Z_{n+1}$};
				\draw [dashed] (0,6) -- (4,6);
				
				%\draw (0,3) node {--};
				\draw (0,3) node[left] {$M_{n+1}$};
				
				\draw (0.875,-1) arc (0:140:0.5);
				\draw [black] (2,0) arc (0:170:01);
				
				\draw (0.5,-0.5) node[above] {$\varepsilon_n$};
				\draw [black] (1.75,0.8) node[above] {$\varepsilon_{n+1}$};
				
				\draw [<->] (1.8,6.15) -- (3.88,6.15);
				\draw (2.9,6.1) node[above] {\tiny $(\varepsilon_{n+1}-\varepsilon_n)\|Z\|$};
			\end{tikzpicture}
			\caption{Illustration of \\
				$\Bn{n+1}\cap\Bbisn{n}$}
			\label{fig: example}
		\end{subfigure}
		\begin{subfigure}{.45\textwidth}
			\centering
			\begin{tikzpicture}[use Hobby shortcut,closed=true]
				\fill [color=gray!40] (3.88,5.9).. (3.6,6.13).. (2.3,6.2).. (1.7,7).. (-1.5,8).. (-1,6.5).. (-2,2).. (-1.7,1).. (-2.2,-1).. (1.4,-1).. (2,1.7).. (2.5,3.5).. (3,5).. (2.3,5.7).. (3.6,5.85)..(3.88,5.9) ;
				%\draw (-6,-6) grid (6,6);
				%(4,5.9)
				%\draw (0,3) node {--};
				\draw (0,3) node[left] {$M_{n+1}$};
				\draw (0,-2) -- (0,8);
				%\draw (0,6) node {--};
				\draw (0,6) node[left] {$\|Z\|$};
				\draw (0.75,-2) --(2,8);
				\draw (1.375,-2) -- (4.5,8);
				%\draw (3.88,6) node {+};
				\filldraw (3.88,6) circle (1pt);
				\draw (4,6) node[right] {$Z_{n+1}$};
				\draw [dashed] (0,6) -- (4,6);
				\draw (0.875,-1) arc (0:140:0.5);
				\draw [black] (2,0) arc (0:170:01);
				\draw (0.5,-0.5) node[above] {$\varepsilon_n$};
				\draw [black] (1.75,0.8) node[above] {$\varepsilon_{n+1}$};
				%\draw [<->] (1.8,6.15) -- (3.88,6.15);
				%\draw (2.9,6.1) node[above] {\tiny $(\varepsilon_{n+1}-\varepsilon_n)\|z\|$};
			\end{tikzpicture}
			\caption{Case of a tentacle  \\
				reaching out to $Z_{n+1}$}
			\label{fig: tentacle}
		\end{subfigure}
	\end{figure}
	Let us first explain the choice of the radius for the ball centered around $z_{n+1}$. This is where the construction of $(M_n)$ and $(\varepsilon_n)$ comes into play. When building the ball around $z_{n+1}$, we need to take a radius small enough to ensure that our ball does not intersect $\mathscr{C}_{\varepsilon_{n}}$ as well as the strip $\mathbb{Z}_{M_n}:=\mathbb{Z} \times \llbracket -M_n, M_n \rrbracket^{d-1}$, since we need to consider only new particles (particles contributing to $A_{n+1}[\infty]\setminus A_{n}[\infty]$).  
	To do it, we define the radius $r_{n+1} = r_{n+1}(l, \varepsilon)$ as 
	\begin{equation}
		\label{eqn: radius expression}
		r_{n+1}=\frac{\varepsilon_{n+1}-\varepsilon_{n}}{2}\|z\|=\frac{\varepsilon}{2^{n+1}}l.
	\end{equation}
	This choice of $r_{n+1}$ ensures that $\mathbb{B}\left(z_{n+1},r_{n+1}\right) \cap \mathscr{C}_{\varepsilon_{n}}=\emptyset$, since $z_{n+1}$ is necessarily at a distance $\left(\varepsilon_{n+1}-\varepsilon_{n}\right)l$ of $\mathscr{C}_{\varepsilon_{n}}$. Moreover, $\mathbb{B}(z_{n+1},r_{n+1}) \cap \mathbb{Z}_{M_n}=\emptyset$, since $M_{n+1} \leq l$ thus $r_{n+1} \leq \|z\|-M_n$.  
	To deal with \eqref{eqn: doublesum}, we write:
	\begin{align}
		&\mathbb{P}\left(Z=z,\ \Bn{n+1}\cap\Bbisn{n} \right)  \nonumber \\
		\leq &\mathbb{P}\left(Z=z,\ \Bn{n+1}\cap\Bbisn{n},\ \#\left(A_{Z^*}\cap \mathbb{B}(z_{n+1},r_{n+1})\right)\leq br_{n+1}^d\right) \label{eqn: tentacle} \\
		+ &\mathbb{P}\left(Z=z,\ \Bn{n+1}\cap\Bbisn{n},\ \#\left(A_{Z^*}\cap \mathbb{B}(z_{n+1},r_{n+1})\right)>br_{n+1}^d\right) \label{eqn: ball}
	\end{align}
	The term \eqref{eqn: tentacle} is handled by Lemma \ref{lemma: tentacle} with $r=r_{n+1}$ as in \eqref{eqn: radius expression}. This gives:
	\begin{multline}
		\mathbb{P}\left(Z=z,\ \Bn{n+1}\cap\Bbisn{n},\ \#\left(A_{Z^*}\cap \mathbb{B}(z_{n+1},r_{n+1})\right)\leq br_{n+1}^d\right) \label{eqn: JLS}\\
		\leq K_0e^{-c_1l^2},
	\end{multline}
	where $c_1=c_1(n, \varepsilon)=\dfrac{c\varepsilon^2}{4^{n+1}}$.  
	\newline
	\begin{sloppypar}
		To deal with \eqref{eqn: ball}, we use the following argument: in order to have more than ${br_{n+1}^d=b\varepsilon^d 2^{-d(n+1)}\|z\|^d}$ \textit{new} particles gathered in a ball around $z_{n+1}$, and knowing only \emph{one} additional particle is thrown from each site, this implies that 
		$\|Z-Z^*\|\geq K\|z\|^{\frac{d}{d-1}}$ (where \comment{$K = K(\varepsilon, n)$} is a positive constant). Indeed, since $\mathcal{H}$ is a hyperplane of dimension $d-1$, for any $\delta > 0$, the number of sources of $\mathbb{B}(Z, \|z\|^{\delta})\cap \mathcal{H}$ is of order $\|z\|^{\delta (d-1)}$. Each of these sources emits only \textit{one} additional particle, so the total number of particles emitted within $\mathbb{B}(Z, \|z\|^{\delta})\cap \mathcal{H}$ is also of order $\|z\|^{\delta (d-1)}$. In order for this quantity to be of the same order as $br_{n+1}^d$, which is of order $\|z\|^d$, it is hence necessary that $\delta = \frac{d}{d-1}$.
		This gives:
	\end{sloppypar} 
	\begin{align*}
		&\mathbb{P}\left(Z=z,\ \Bn{n+1}\cap\Bbisn{n},\ \#\left(A_{Z^*}\cap \mathbb{B}(z_{n+1},r_{n+1})\right)>br_{n+1}^d\right)\\ 
		\leq & \mathbb{P}\left(Z=z,\ \|Z^*-z\|\geq Kl^{\frac{d}{d-1}},\ \Bn{n+1}\cap\Bbisn{n}\right)\\
		\leq & \mathbb{P}\left(\bigcup_{h\geq Kl^{\frac{d}{d-1}}}\bigcup_{\|z'-z\|=h}\Bigl\{\text{the particle sent from } z' \text{ reaches level } l \text{ while staying within } \mathscr{C}_{\varepsilon_{n+1}}\Bigr\} \right) \\
		\leq & \sum_{h\geq Kl^{\frac{d}{d-1}}}\sum_{{\|z'-z\|}=h}\mathbb{P}\left(\text{the particle sent from } z' \text{ reaches level } l \text{ while staying within } \mathscr{C}_{\varepsilon_{n+1}} \right).	
	\end{align*} 
	Proceeding in the same way as Section \ref{section: donut method}, we can build donuts between levels $h$ and $l$ within the cone $\mathscr{C}_{\varepsilon_{n+1}}$. The previous probability is then handled using a donut argument \comment{as in Proposition \ref{prop: crossing prob}}, and is smaller than $(1-c)^k$, with $c=(2d)^{-2}$ and with $k = k(h,l,\varepsilon_{n+1})$ such that
	\begin{equation}
		\label{eqn: k}
		k\geq\dfrac{-1}{2\log(1-2\varepsilon_{n+1})}\times \log\left(\dfrac{h}{l}\right),
	\end{equation}
	given that $\varepsilon$ is small enough \comment{and $M$ is large enough} (here, we used the fact that $z_{n+1}$ is at the same level as $z$). Therefore,
	\begin{multline*}
		\mathbb{P}\left(Z=z,\ \Bn{n+1}\cap\Bbisn{n},\ \#\left(A_{Z^*}\cap \mathbb{B}(z_{n+1},r_{n+1})\right)>br_{n+1}^d\right)\\ 
		\leq \sum_{h\geq Kl^{\frac{d}{d-1}}}\sum_{\|z'-z\|=h} (1-c)^k.
		%\leq K_{d,\varepsilon, n}l^{d-\frac{C}{d-1}}.
	\end{multline*} 
	Now, using \eqref{eqn: k}, standard computations yield 
	\[
	\sum_{h\geq Kl^{\frac{d}{d-1}}}\sum_{\|z'-z\|=h} (1-c)^k \leq \comment{K_{d,\varepsilon, n}}l^{d-\frac{C}{d-1}},
	\]
	where $C:=C(\varepsilon_{n+1})=\frac{\log(1-c)}{2\log(1-2\varepsilon_{n+1})}$ can be taken as large as we want (given once again that $\varepsilon$ is small enough) and \comment{$K_{d,\varepsilon, n}$ denotes a positive constant depending only on $d,\ \varepsilon$ and $n$}.
	Combining this with \eqref{eqn: JLS}, we get:
	\begin{multline*}
		\mathbb{P}\left(\Bn{n+1}\cap\Bbisn{n}\right) \\
		\begin{split}
			& \leq\sum_{l\geq M_{n+1}}\sum_{\|z\|=l} K_0 e^{-c_1l^2}
			+\sum_{l\geq M_{n+1}}\sum_{\|z\|=l} \comment{K_{d,\varepsilon, n}}l^{d-\frac{C}{d-1}} \nonumber \\
			&\leq K_d\sum_{l\geq M_{n+1}} l^{d-2}e^{-c_1l^2}
			+\comment{K_{d,\varepsilon, n}}\sum_{l\geq M_{n+1}} l^{d-2}l^{d-\frac{C}{d-1}}. 
		\end{split}
		\label{eqn: final eqn}
	\end{multline*}
	Notice that the first term of the previous sum can be bounded by $Ke^{-\frac{c_1M_{n+1}^2}{2}}$ for some constant $K$. Since $M\leq M_{n+1}$, it is clear that $Ke^{-\frac{c_1M_{n+1}^2}{2}}\leq \frac{C'_{\varepsilon,n}}{M^L}$ for some constant $C'_{\varepsilon,n}>0$.
	To deal with the second term, recall that $C:=C(\varepsilon_{n+1})$ can be taken as large as necessary (by taking $\varepsilon$ sufficiently small). We can therefore choose $\varepsilon$ small enough such that:
	\begin{equation*}
		\label{eqn= finite sum 2}
		\sum_{l\geq M_{n+1}}l^{2d-2-\frac{C}{d-1}}\leq \dfrac{C''_{\varepsilon,n}}{M^L}.
	\end{equation*}
	Recall that from our induction hypothesis, $\mathbb{P}\left(\Bn{n}\right)\leq \dfrac{C_{\varepsilon,n}}{M^L}$. 
	This implies
	\begin{multline*}
		\begin{split}
			\mathbb{P}\left(\Bn{n+1}\right)&\leq \mathbb{P}\left(\Bn{n+1}\cap \Bbisn{n}\right)+\mathbb{P}\left(\Bn{n}\right)\\
			&\leq \dfrac{C'_{\varepsilon,n}}{M^L} + \dfrac{C''_{\varepsilon,n}}{M^L} + \dfrac{C_{\varepsilon,n}}{M^L} \\
		\end{split}
	\end{multline*}
	and concludes the proof of Proposition \ref{prop: BMN}. 
\end{proof}

\section{Strong stabilization}
\label{section: strong stab}
This section is devoted to the proof of Theorem \ref{thm: strong stab} which heavily lies on the donut method and the global upper bound (see Sections \ref{section: donut method} and \ref{section: global upper bound}).

Fix an integer $\alpha \geq 2$. 
For $M\geq 1,\ j\geq 0$, recall the definition of : 
\[
\mathrm{Ann}(M,j) := \mathcal{H} \cap \big( \mathbb{B}((j+2)M^\alpha) \setminus \mathbb{B}((j+1)M^\alpha ) \big), \; j \geq 0,
\] 
and let $E_{M,j}$ be the following event:
\[
E_{M,j} = \left\lbrace\begin{array}{c}
	\mbox{ At least one of the particles starting from $\mathrm{Ann}(M,j)$  }\\
	\mbox{ visits the strip $\mathbb{Z}_{M}$ before exiting the current aggregate}\\
\end{array}
\right\rbrace .
\]
According to the Borel-Cantelli lemma, it is sufficient to show that 
\[
\sum_{M\geq 1}\mathbb{P}\left(\bigcup_{j\geq 0}E_{M,j}\right)<+\infty.
\]
To do it, we write: 
\[
\mathbb{P}\left(\bigcup_{j\geq 0}E_{M,j}\right)\leq \sum_{j\geq 0}\mathbb{P}\left(E_{M,j} \cap \Bbis{n} \right) + \mathbb{P}\left(\B{n}\right).
\]
We focus on the left-hand term. 
\comment{To deal with it}, let $j\geq 0$, take $l=M^{\alpha}(j+1)$ and take $k = k(l,M,\varepsilon)$ as in \eqref{eqn: value of k}. Define $N_{tot}=N_{tot}(n,M,j)$ as the total number of particles sent from $\mathrm{Ann}(M,j)$. We have:
\begin{multline*}
	\mathbb{P}\left(E_{M,j} \cap \Bbis{n}\right)\\
	\begin{split}
		&=\mathbb{P}\left(\bigcup_{i=1}^{N_{tot}}\Bigl\{\text{particle }i \text{ visits } \mathbb{Z}_{M} \text{ before exiting the aggregate}\Bigr\}\cap \Bbis{n}\right)\\
		&\leq \sum_{i=1}^{N_{tot}}\mathbb{P}\left(\Bigl\{\text{particle } i \text{ visits } \mathbb{Z}_{M} \text{ before exiting the aggregate} \Bigr\}\cap \Bbis{n}\right)\\
		&\leq \sum_{i=1}^{N_{tot}}\mathbb{P}\left(\Bigl\{\text{particle } i \text{ crosses } k \text{ donuts before exiting the aggregate} \Bigr\}\cap \Bbis{n}\right)\\
		&\leq \sum_{i=1}^{N_{tot}}\mathbb{P}\left(\Bigl\{\text{the \emph{walk} associated with particle } i \text{ crosses } k \text{ donuts before exiting } \mathscr{C}_{\varepsilon}\cap \mathbb{Z}_M^c \Bigr\}\right)\\
		&\leq N_{tot}(1-c)^k,\\
	\end{split}
\end{multline*}
where the last line comes from Proposition \ref{prop: crossing prob}.
\newline
Notice here that the global upper bound $\Bbis{n}$ is used to deduce two different arguments. The first time, it allows us to say that if a particle reaches $\mathbb{Z}_M$ before exiting the aggregate, then it necessarily crosses the $k$ donuts $\textbf{D}^0,\dots,\ \textbf{D}^{k-1}$ before exiting the aggregate, since the aggregate is contained within the cone, which itself is contained within the union of the donuts. 
The second time, it allows us to say that if a particle crosses $k$ donuts without exiting the aggregate, then in particular, it does so without exiting the cone, and the same is true for its associated walk.
Now, using \eqref{eqn: value of k} with $l=M^{\alpha}(j+1)$, we get:
\begingroup
\allowdisplaybreaks
\begin{align*}
	N_{tot}(1-c)^k & \leq N_{tot}\exp\left(K(\varepsilon) \log\left(\dfrac{M^{\alpha}(j+1)}{M}\right)\log(1-c)\right)\\
	&\leq N_{tot}M^{C(1-\alpha)}(j+1)^{-C},
\end{align*}
\endgroup
where $C:=-K(\varepsilon)\log(1-c)=\dfrac{\log(1-c)}{2\log(1-2\varepsilon)}$ can be taken arbitrarily large, by taking $\varepsilon$ arbitrarily small.
\newline
Since $N_{tot}\leq K_dM^{\alpha (d-1)}{j^{d-2}n}$, we have: 
\begin{multline*}
	\sum_{M\geq 1}\mathbb{P}\left(\bigcup_{j\geq 0}E_{M,j}\right)\\
	\begin{split} &\leq	\sum_{M\geq 1}\sum_{j\geq 0}\mathbb{P}\left(E_{M,j}\cap \Bbis{n}\right)+\sum_{M\geq 1}\mathbb{P}\left(\B{n}\right)\\
		&\leq K_dn\sum_{M\geq 1}M^{C(1-\alpha)+\alpha (d-1)}\sum_{j\geq 0}{j^{d-2}}(j+1)^{-C}+\sum_{M\geq 1}\mathbb{P}\left(\B{n}\right).
	\end{split}\\
\end{multline*}
The left hand-term of the sum is finite since $\alpha>1$ and since we can pick $C=C(\varepsilon)$ sufficiently large, while the second term is handled using Theorem \ref{prop: BM}. This concludes the proof of Theorem \ref{thm: strong stab}.

\section{Shape theorem}
\label{section: shape thm}

In this section, we prove Theorem \ref{thm: shape} following the same strategy as \cite{asselah2013sublogarithmic, asselah2019outer} and \cite{chenavier2023bi}, by splitting the proof into two parts: the lower bound and the upper bound. We begin by showing the lower bound, as we will be using it later for the proof of the upper bound. 

Let us recall that, for any real number $x>0$, the slab $\mathcal{R}_x$ and the strip $\mathbb{Z}_x$ are defined as \[\mathcal{R}_x = \llbracket -\lfloor x \rfloor , \lfloor x \rfloor \rrbracket \times \mathbb{Z}^{d-1}\]  and \[\mathbb{Z}_x = \mathbb{Z} \times \llbracket -\lfloor x \rfloor , \lfloor x \rfloor \rrbracket^{d-1}.\]

\subsection{Proof for the lower bound}
\label{subsection: the lower bound}

In this section, we show the lower bound of Theorem \ref{thm: shape}, which is the following result: for any integers $d\geq 3$ and any $\alpha\geq 1$, there exists a constant $C = C(d,\alpha)>0$ such that, almost surely, there exists $N\geq 1$ such that for any $n\geq N$,
\[
\mathcal{R}_{n/2-C\sqrt{\log n}}\cap \mathbb{Z}_{n^\alpha}\subset \An \cap \mathbb{Z}_{n^\alpha}.
\]
We adopt in this section the notation of \cite{asselah2013logarithmic,asselah2013sublogarithmic,chenavier2023bi} and denote by $A(\eta)$ the aggregate generated by an initial configuration $\eta$. Since we are launching $n$ particles from each site of $\mathcal{H}$, we will mostly be using the notation $A(n\mathbbm{1}_{\mathcal{H}})$ to refer to $\An$. 

For $k\in\mathbb{N}$, we define the shell $S_k$ by
\[
S_k=\left(\mathcal{R}_{(k+1)\sqrt{\log n}}\setminus \mathcal{R}_{k\sqrt{\log n}}\right)\cap \mathbb{Z}_{n^\alpha}.
\] 
We let
\[
\partial \mathcal{R}_{k\sqrt{\log n}}=\{ -\lfloor k\sqrt{\log n} \rfloor , \lfloor k\sqrt{\log n} \rfloor \}\times \mathbb{Z}^{d-1}.
\]
and write 
\[
\partial_{k,n}=\partial\Rk \cap \mathbb{Z}_{n^{\alpha}}.
\]
Now, for $z \in \partial \mathcal{R}_{k\sqrt{\log n}}$ we define the \emph{tile} and \emph{cell} centered in $z$ as
\[
\tau(z)=\mathbb{B}\left(z,\frac{\sqrt{\log n}}{2}\right) \cap \partial \mathcal{R}_{k\sqrt{\log n}} \quad \text{ and } \quad \mathcal{C}(z)=\mathbb{B}\left(z,\sqrt{\log n }\right)\cap \mathcal{R}_{k\sqrt{\log n}}^c.
\]

The strategy to prove the lower bound is to show that each tile of $\Rk$ is likely visited by many particles, and then show that if many particles reached a tile, they are likely to fill up the corresponding cell. The idea here is similar to that of a floodgate. Each tile $\tau$ can be seen as a floodgate, and the particles as water. We stop the particles once they reach $\tau$, and let them accumulate on the tile, just like a floodgate would store water. Then, when there is a sufficient number of particles accumulated on $\tau$, we release the particles and show that they are likely to fill up the corresponding cell. This is the same as opening the floodgates and letting the water run free again. 
For some configuration $\eta$ and $B\subset \Rk$, we will denote by:
\begin{itemize}
	\item 	$W_{k\sqrt{\log n}}(\eta,B)$ the number of particles with initial configuration $\eta$ that hit set $B$ before or when exiting $\Rk$ \comment{(and before leaving the aggregate)}.
	\item $M_{k\sqrt{\log n}}(\eta,B)$ the number of random walks with initial configuration $\eta$ that hit set $B$ before or when exiting $\Rk$.\newline
\end{itemize}
We say that set $B$ is not covered if $B\nsubset A(n\mathbbm{1}_{\mathcal{H}})$. It is sufficient to show that there exists a constant $C$ such that for all $L>1,\ n\geq 1$ and $k\leq \frac{n}{2\sqrt{\log n}}-C$, we have 
\begin{equation}
	\mathbb{P}\left(S_k \text{ is not covered }\right) \leq \frac{c}{n^L}.
\end{equation}
\comment{This implies that
	\[
	\sum_{n\geq 1}\mathbb{P}\left(S_k \text{ is not covered }\right)\leq \sum_{n\geq 1}\frac{c}{n^L}<+\infty,
	\]
	which, combined with the Borel-Cantelli lemma, suffices to prove the lower bound.}
Now, let
\[
\mathcal{T}_{k\sqrt{\log n}}=\{\tau(z),\ z \in \partial_{k,n}\},
\]
and for a tile $\tau=\tau(z) \in \mathcal{T}_{k\sqrt{\log n}}$,
\[
\mu(\tau)=\mathbb{E}\left[\Mk(n\mathbbm{1}_\mathcal{H},\tau)\right]-\mathbb{E}\left[\Mk\left(\Rk\setminus\mathcal{Z},\tau\right)\right],
\]
where $\mathcal{Z}$ is the following set:
\[
\mathcal{Z}=\{z'\in \Rk: d(z',\tau(z))\leq b\sqrt{\log n}\}
\] 
for some $b>0$ which will be chosen later, and with $d(x,A)= \displaystyle{\inf_{y\in A}} \|x-y\|$. Now, fix $C>0$. For $k\leq \frac{n}{2\sqrt{\log n}}-C$, we write:
\begin{equation}
	\label{eqn: few tile}
	\mathbb{P}\left(S_k \text{ is not covered } \right) \leq  \mathbb{P}\left(\exists\tau\in\Tauk,\ \Wk(n\mathbbm{1}_\mathcal{H},\tau) <\frac{1}{3}\mu(\tau)\right)
\end{equation}
\begin{equation}
	\label{eqn: many tile}
	+ \mathbb{P}\left(\forall\tau\in\Tauk,\ \Wk(n\mathbbm{1}_\mathcal{H},\tau)\geq\frac{1}{3}\mu(\tau),\ S_k \text{ is not covered }\right) \comment{.}
\end{equation}

The second term of expression \eqref{eqn: many tile} will be handled later in Section \ref{subsubsection: handling}. To handle \eqref{eqn: few tile}, essentially, we must control the probability that few particles hit a tile. This is achieved using Lemma \ref{lemma: 6.2} below. Notice that working on strip $\mathbb{Z}_{n^{\alpha}}$ allows us to use a union bound on $\mathbb{P}\left(\exists\tau\in\Tauk,\ \Wk(n\mathbbm{1}_\mathcal{H},\tau) <\frac{1}{3}\mu(\tau)\right)$, since 
\[
\#\Tauk \leq 2(2n^{\alpha}+1)^{d-1}\leq Kn^{\alpha(d-1)}.
\]
It is sufficient to prove \comment{that for any tile} $\tau \in \Tauk$,
\comment{
	\begin{equation}
		\label{eqn: few particles}
		\mathbb{P}\left(\Wk(n\mathbbm{1}_\mathcal{H},\tau) <\frac{1}{3}\mu(\tau)\right)\leq \exp\left( -\kappa C^2\log(n) \right).
	\end{equation}
}
The above inequality is a consequence of the following lemma (Lemma 2.5 of \cite{asselah2013sublogarithmic}):
\begin{lemma} %(\textit{Asselah, Gaudillière})
	\label{lemma: 6.2}
	Suppose that a sequence of random variables $\{W_n, M_n,L_n,\widetilde{M}_n; n\geq 0\}$ and a sequence of real numbers $(c_n)_{n\geq 0}$ satisfy for any $n\geq 0$:
	\begin{equation*}
		W_n+L_n+c_n\geq \widetilde{M}_n\qquad\mbox{and}\qquad \widetilde{M}_n\overset{\mathrm{law}}{=}M_n.
	\end{equation*}
	Assume that $W_n$ and $L_n$ are independent and that $L_n$ and $M_n$ are both sums of independent Bernoulli random variables with finite first moment. Assume also that 
	\begin{enumerate}
		\item[\bf (H1)]\label{H1} the Bernoulli variables $\{Y_1^{(n)},\ Y_{2}^{(n)}, \dots\}$ whose series is $L_n$ satisfy for some ${\comment{\eta}>1}$:
		\[
		\underset{n}{\sup}~\underset{i}{\sup}\,\mathbb{E}\left[{Y_i^{(n)}}\right]< \comment{\dfrac{\eta -1}{\eta}};
		\]
		\item[\bf (H2)]\label{H2} $\mu_n=\mathbb{E}\left[{M_n-L_n}\right]\geq 0$.
	\end{enumerate}
	Then, for any $n\geq 0$ and $\xi_n\in\R$, we have for any $\lambda\geq 0$,
	\begin{equation*}
		\mathbb{P}\left({W_n<\xi_n}\right)\leq \exp\left(-\lambda\left(\mu_n-\xi_n-c_n\right)+\frac{\lambda^2}{2}\left(\mu_n+\comment{\eta}\sum_{i=1}^{\infty}\mathbb{E}\left[Y_i^{(n)}\right]^2\right)\right).
	\end{equation*}
\end{lemma}

We first need to check that both hypotheses of Lemma \ref{lemma: 6.2} are satisfied. To do so, we use a similar strategy of \cite{lawler1992internal} and \cite{asselah2013sublogarithmic} to stochastically dominate the variable $\Mk(n\mathbbm{1}_{\mathcal{H}},\tau)$.
We have:
\begin{equation}
	\label{eqn: stoch eq}
	\Wk(\init, \tau) + \Mk\left(A_{k\sqrt{\log n}}(\init),\tau\right) \overset{\mathrm{law}}{=} \Mk(\init,\tau),
\end{equation}
where $A_{k\sqrt{\log n}}(\init)=A(\init)\cap\Rk$.

We explain the general idea behind \eqref{eqn: stoch eq}. Consider a random walk starting on some site of $\mathcal{H}$ and hitting $\tau$ when exiting $\Rk$. Such a walk is accounted for in $\Mk$. Now, it is possible for the particle associated to that random walk to also hit $\tau$ when exiting $\Rk$. In that case, it is accounted for in $\Wk$. If, however, it settles beforehand, say on some site $y \in \Rk$, then we can find a coupling such that the trajectory of the walk starting after the particle has settled is the same as the trajectory of a random walk started on $y$ hitting $\tau$ when exiting $\Rk$. Such a term is accounted for in $\Mk\left(A_{k\sqrt{\log n}}(\init),\tau\right)$. This stochastic equality will be of use when applying Lemma \ref{lemma: 6.2}, using it with $M_n=\Mk(\init,\tau)$ and 
\[
\widetilde{M}_n=\widetilde{M}_{k\sqrt{\log n}}(\init,\tau):=\Wk(\init, \tau) + \Mk\left(A_{k\sqrt{\log n}}(\init),\tau\right).
\]
Now, simply using the fact that $A_{k\sqrt{\log n}}(\init)\subset \Rk$, we have:
\begin{equation}
	\label{eqn: stoch dom}
	\MMk(\init,\tau) \leq \Wk(\init, \tau) + \Mk\left(\Rk,\tau\right) \quad \mathrm{a.s.}
\end{equation}
Now, we let $c_n=\#\mathcal{Z}\leq c\left(\sqrt{\log n}\right)^d$, \comment{where $c=c(d, b)>0$}. Using \eqref{eqn: stoch dom} gives:
\[
\MMk(\init,\tau)\leq \Wk(\init, \tau) + \Mk\left(\Rk\setminus\mathcal{Z},\tau\right)+ c_n.
\]
Note that this inequality is similar to the one in Lemma \ref{lemma: 6.2}. The idea will be to apply Lemma \ref{lemma: 6.2} with
\[
\begin{cases}
	W_n=\Wk(\init,\tau) \\
	M_n=\Mk(\init,\tau) \\
	\tilde{M}_n = \MMk(\init,\tau) \\
	L_n= \Mk\left(\Rk\setminus\mathcal{Z},\tau\right)\\
\end{cases}
\]
Note that both $L_n$ and $M_n$ can be written as sums of independent Bernoulli random variables, as:
\[
L_n=\sum_{y\in \Rk\setminus\mathcal{Z} } \mathbbm{1}^y_{S(H_{k\sqrt{\log n}})\in \tau} \quad \mathrm{ and } \quad
M_n=\sum_{y\in n\mathbbm{1}_{\mathcal{H}}} \mathbbm{1}^y_{S(H_{k\sqrt{\log n}})\in \tau},
\]
where the indicators $\mathbbm{1}^y$ correspond to independent simple symmetric random walks beginning at $y$, and $H_{k\sqrt{\log n}}$ denotes the hitting time of $\mathcal{R}_{k\sqrt{\log n}}$.
Before applying Lemma \ref{lemma: 6.2}, we must ensure that hypotheses $\textbf{(H1)}$ and $\textbf{(H2)}$ hold. This is done in the next subsection.
\subsubsection{Checking hypotheses \textbf{(H1)} and \textbf{(H2)}}
\label{subsubsection: checking hypotheses H1 and H2}
Let us begin by giving the following lemma, which ensures hypothesis $\textbf{(H1)}$ of Lemma \ref{lemma: 6.2} for some $b>0$ in the definition of $\mathcal{Z}$. This lemma is analogous to Lemma 5.1 of \cite{asselah2013logarithmic}, and is adapted to the case of \emph{slabs}. We omit its proof.
\begin{lemma}
	\label{lemma 6.3}
	There exists a positive constant $\kappa_0>0$ such that for all $r>0$, for any $y\in \mathcal{R}_r$ and $x\in\partial \mathcal{R}_r\setminus\{y\}$, we have	
	\[  
	\mathbb{P}_y\left(S(H_r)=x\right)\leq \frac{\kappa_0}{\|x-y\|^{d-1}},
	\]
	where $H_r$ denotes the hitting time of $\partial \mathcal{R}_r$ for the simple random walk $(S(t))_{t\geq 0}$.
\end{lemma}
This ensures $\textbf{(H1)}$, since for all $y\in \Rk\setminus\mathcal{Z}$,
\begin{align*}
	\mathbb{E}\left[\mathbbm{1}^y_{S(H_{k\sqrt{\log n}})\in \tau}\right]&= \sum_{x\in\tau} \mathbb{P}_y\left(S(H_{k\sqrt{\log n}})\comment{=x}\right)\\
	&\leq \comment{\sum_{x \in \tau}} \dfrac{\kappa_0}{\|x-y\|^{d-1}}\\
	&\leq \#\tau \dfrac{\kappa_0}{(b\sqrt{\log n})^{d-1}} \\
	&\leq \dfrac{c}{b^{d-1}},
\end{align*}
where the last line comes from the fact that $\#\tau$ is of order $(\sqrt{\log n})^{d-1}$. Thus taking $b$ large enough in the definition of $\mathcal{Z}$ ensures $\textbf{(H1)}$.
We now show that hypothesis $\textbf{(H2)}$ holds as well. We show that if $\tau$ is a tile at a distance $C\sqrt{\log n}$ of $\partial \mathcal{R}_{n/2}$, then for some positive constant $\kappa>0$, we have 
\begin{equation}
	\label{eqn: lower bound mu}
	\mu(\tau)\geq \kappa C\log (n)^{d/2}.
\end{equation}
We also \comment{show} the following:
\begin{equation}
	\label{eqn: E Yi}
	\sum_{y\in\Rk \setminus\mathcal{Z}}\mathbb{P}_y\left(S(H_{k\sqrt{\log n}})\in\tau\right)^2\leq c\log(n)^{d-1}.
\end{equation}
The proofs of \eqref{eqn: lower bound mu} and \eqref{eqn: E Yi} are given in Section \ref{sec: shape thm proofs}.
\subsubsection{Application of Lemma \ref{lemma: 6.2}}
\label{subsubsection: application of lemma}
We now have all the tools in hand to prove \comment{\eqref{eqn: few particles}}. To do so, we use the previous results and Lemma \ref{lemma: 6.2}. \comment{Notice that $C=C(b)$ can be taken large enough such that \eqref{eqn: lower bound mu} gives $\mu(\tau) \geq 3c_n$.} Therefore, applying Lemma \ref{lemma: 6.2} gives for all $\lambda\geq 0$:
\begin{equation}
	\label{eqn: optimization}
	\mathbb{P}\left(\Wk(\init,\tau)<\frac{1}{3}\mu(\tau)\right)\leq \exp\left( -\frac{\lambda}{3}\mu(\tau) +\frac{\lambda^2}{2}\left(\mu(\tau)+\comment{\eta} c\log (n)^{d-1}\right) \right).
\end{equation}
After optimizing in $\lambda$, and using that $\mu(\tau)\geq \kappa C\log (n)^{d/2}$ we get that 
\[
\mathbb{P}\left(\Wk(\init,\tau)<\frac{1}{3}\mu(\tau)\right)\leq \exp\left( \comment{-\kappa' C^2\log(n)} \right),
\]
\comment{for some constant $\kappa' >0$.}
Therefore, 
\begin{align*}
	\mathbb{P}\left(\exists\tau\in\Tauk,\ \Wk(n\mathbbm{1}_\mathcal{H},\tau) <\frac{1}{3}\mu(\tau)\right)&\leq \#\Tauk \exp\left( \comment{-\kappa' C^2\log(n)}\right) \\
	&\leq Kn^{\alpha(d-1)}\exp\left( \comment{-\kappa' C^2\log(n)}\right),
\end{align*}
which for $C$ large enough decreases faster than any power of $n^{-1}$.
The proof of the optimization is given in Section \ref{sec: lower bound proofs}.
\subsubsection{Handling \eqref{eqn: many tile}}
\label{subsubsection: handling}
We now focus on giving an upper bound of \eqref{eqn: many tile}. The idea here is that if many particles hit a tile, they are very likely to fill the corresponding cell. Note that we have the following inclusion
\[
S_k\subset \bigcup_{z\in\partial_{k,n}} \mathcal{C}(z),
\]
\comment{where $\partial_{k,n}=\partial\Rk \cap \mathbb{Z}_{n^{\alpha}}$.}
Now, using the previous inclusion and the fact that $\mu(\tau)\geq \kappa C\log (n)^{d/2}$, we get
\begin{multline*}
	\mathbb{P}\left(\forall\tau\in\Tauk,\ \Wk(n\mathbbm{1}_\mathcal{H},\tau)\geq\frac{1}{3}\mu(\tau),\ S_k \text{ is not covered }\right) \\
	\begin{split}
		&\leq  \mathbb{P}\left(\bigcup_{z\in\partial_{k,n}} \mathcal{C}(z) \text{ is not covered } ,\ \forall\tau\in\Tauk,\ \Wk(n\mathbbm{1}_\mathcal{H},\tau)\geq\frac{\kappa C}{3}\log (n)^{d/2}\right) \\
		&\leq \sum_{z\in\partial_{k,n}}\mathbb{P}\left(\mathcal{C}(z) \text{ is not covered } \;\middle|\; \forall\tau\in\Tauk,\ \Wk(n\mathbbm{1}_\mathcal{H},\tau)\geq\frac{\kappa C}{3}\log (n)^{d/2} \right).
	\end{split}
\end{multline*}
The following lemma shows that if many particles are initially contained inside a ball of radius $R/2$, they have a high probability of filling the ball of radius $R$. \comment{(See Lemma 1.3 of \cite{asselah2013sublogarithmic}, which describes the idea mentioned above of floodgates being open and filling up a given area).}
\begin{lemma}
	\label{lemma: 1.3} 
	Choose $R$ and $A$ large enough. Assume that $\lfloor AR^d \rfloor$ particles lie initially on $\mathbb{B}(0,R/2)$. We call $\eta$ the initial configuration of these particles and $A(\eta)$ the aggregate they produce. There are positive constants $\{\kappa_d,\ d\geq 3\}$ independent of $R$ and $A$ such that 
	\[
	\mathbb{P}\left(\mathbb{B}(0,R)\not \subset A(\eta)\right)\leq \exp\left(-\kappa_d AR^2\right).
	\]
\end{lemma}
In our case, we know that each tile has been hit with at least $\frac{\kappa C}{3}\log (n)^{d/2}$ particles. Now, recall that a cell's radius is twice the radius of a tile, meaning we can apply Lemma \ref{lemma: 1.3} directly, and get
\begin{multline*}
	\mathbb{P}\left(\mathcal{C}(z) \text{ is not covered } \;\middle|\; \forall\tau\in\Tauk,\ \Wk(n\mathbbm{1}_\mathcal{H},\tau)\geq\frac{\kappa C}{3}\log (n)^{d/2} \right)\\
	\leq \exp\left(-\kappa_d C \log(n)\right).
\end{multline*}
Now, using that 
\[
\# \partial_{k,n}=\#\left(\partial\Rk \cap \mathbb{Z}_{n^{\alpha}}\right)=2\times \left(2n^{\alpha}+1\right)^{d-1}\leq Kn^{\alpha(d-1)}
\]
we can conclude on the upper bound of \eqref{eqn: many tile}. We have 
\begin{multline*}
	\mathbb{P}\left(\forall\tau\in\Tauk,\ \Wk(n\mathbbm{1}_\mathcal{H},\tau)\geq\frac{1}{3}\mu(\tau),\ S_k \text{ is not covered }\right)\\
	\leq Kn^{\alpha(d-1)}\exp\left(-\kappa_d C \log(n)\right)
\end{multline*}
which decreases, for $C$ large enough, faster than any given power of $n^{-1}$. This concludes the proof of the lower bound.

We understand thanks to Lemma \ref{lemma: 1.3} the source of the sublogarithmic fluctuations for the aggregate. We use Lemma \ref{lemma: 1.3} in hopes of applying the Borel-Cantelli Lemma and getting a bound decreasing faster than any power of $n$. The bound we obtain is of order $\exp\left(-\kappa_d AR^2\right)$, hence choosing $R$ of order $\sqrt{\log n}$ means $\exp\left(-\kappa_d AR^2\right)$ \comment{is of order $n^{-\kappa A}$, with $\kappa>0$ and $A = \Theta(1)$ (\textit{i.e.} if there exist constants such that $c_1 \leq A \leq c_2$)}. We therefore pick shells of size $R \approx \sqrt{\log n}$, leading to sublogarithmic fluctuations. Note that the bound in Lemma \ref{lemma: 1.3} for $d=2$ is equal to $\exp\left(-\kappa_2 \frac{AR^2}{\log R}\right)$, so choosing $R\approx \sqrt{\log n}$ no longer grants a functional bound for the Borel-Cantelli Lemma.
\subsection{Proof for the upper bound}
\label{sub: the upper bound}
In this section, we prove the upper bound of Theorem \ref{thm: shape}, that is: for any integers $d\geq 3$ and $\alpha \geq 1$, there exists $C>0$ such that, almost surely, there exists $N\geq 1$ such that for any $n\geq N$,
\[	
\An \cap \mathbb{Z}_{n^\alpha} \subset \mathcal{R}_{n/2+C\sqrt{\log n}}\cap \mathbb{Z}_{n^\alpha}.
\]
The proof follows the same lines as in \cite{chenavier2023bi}. Fix $\alpha\geq 1$ and take $C>0$. 

%\comment{From the proof of Theorem \ref{thm: strong stab}, we can show that for any $L>0$, for $\alpha > 2d$ and for $n$ large enough,}
%\begin{equation}
%	\label{eqn: stab in n}
%	\mathbb{P}\left(A(n\mathbbm{1}_{\mathcal{H}})\cap\mathbb{Z}_{n^{\alpha}}\neq \Ang\cap\mathbb{Z}_{n^{\alpha}}\right)\leq n^{-L},
%\end{equation}
%\comment{for $\gamma > \alpha$.}
%\comment{This result is not immediate, and requires us to slightly tweak the proof of Theorem \ref{prop: BM}. A detailed description of this is given in Appendix \ref{appendix}. Note as well that we prove this result for $\alpha > 2d$, but this isn't an issue, as the result follows for smaller values of $\alpha$ \textit{a fortiori}.}
%
%The approximation given by \eqref{eqn: stab in n} is essential to the rest of the proof, because instead of considering an infinite number of particles in $A(n\mathbbm{1}_{\mathcal{H}})$, we reduce the problem to a finite number of particles in $\Ang$. 
From the Borel-Cantelli lemma, it is sufficient to prove that 
\begin{equation}
	\label{eqn: upper bound}
	\mathbb{P}\left( \An \cap \mathbb{Z}_{n^{\alpha}} \not \subset \mathcal{R}_{n/2+C\sqrt{\log n}}\cap \mathbb{Z}_{n^\alpha} \right)
\end{equation}
is smaller than any power of $n^{-1}$. To do so, we define the random variable 
\[
X(n)=\max\left\{|z_1|,\ z\in \An\cap \mathbb{Z}_{n^{\alpha}}\right\}.
\]
Now, notice that we can bound \eqref{eqn: upper bound} by $\mathbb{P}\left(X(n)> \frac{n}{2}+C\sqrt{\log n} \right)$.
%Since ${\#\mathcal{H}_{n^{\gamma}}\leq \left(2n^{\gamma}+1\right)^{d-1}}$, we know that $X(n)\leq n\left(2n^{\gamma}+1\right)^{d-1}$ a.s. Taking the supremum over the point $z\in\mathbb{Z}_{n^\alpha}\cap \{z:\ \frac{n}{2}+C\sqrt{\log n} \leq |z_1|\leq n\left(2n^{\gamma}+1\right)^{d-1}\}$, it suffices to prove that 
%\[
%\sup_z\mathbb{P}\left(z\in A(n\mathbbm{1}_{\mathcal{H}_{n^{\gamma}}}),\ |z_1|=X(n)\right)
%\]
%is lower than any power of $n^{-1}$. 
Hence, we write :
\begin{align*}
	\mathbb{P}\left( X(n) >\frac{n}{2} + C\sqrt{\log n} \right) & \leq \mathbb{P}\left( \exists z \in \An \cap \Z_{n^{\alpha}},\ X(n) = |z_1| >\frac{n}{2} + C\sqrt{\log n} \right)\\
	& \leq \sum_{z \in \Z_{n^{\alpha}}} \mathbb{P}\left(z \in \An,\ X(n) = |z_1| >\frac{n}{2} + C\sqrt{\log n} \right).
\end{align*}
Now, since $\An$ is translation-invariant (with respect to $T_k,\ k \in \mathcal{H}$), this quantity only depends on the first coordinate of $z$. For symmetry reasons, we can focus only on the case where $z_1 \geq 0$. So for $l\geq 0$, we define $z(l)$ as $z(l):= (l + \frac{n}{2},0, \dots, 0)$, and write:
\begin{multline*}
	\mathbb{P}\left( X(n) >\frac{n}{2} + C\sqrt{\log n} \right)\\
	\begin{split}
		&\leq 2(2n^{\alpha}+1)^{d-1}\mathbb{P}\left( \exists l > C\sqrt{\log n},\ z(l) \in \An ,\ X(n) = l+\frac{n}{2} \right) \\
		& \leq 2(2n^{\alpha}+1)^{d-1}\sum_{l \geq C\sqrt{\log n}}\mathbb{P}\left( z(l) \in \An ,\ X(n) = l+\frac{n}{2} \right).
	\end{split}
\end{multline*}
Now, fix $l \geq C\sqrt{\log(n)}$.
%Now, fix $z\in \mathbb{Z}_{n^{\alpha}}$ such that :
%\[
%\frac{n}{2}+C\sqrt{\log n}\leq |z_1|\leq n\left(2n^{\gamma}+1\right)^{d-1}.
%\]
%We define $h(n)=|z_1|-\frac{n}{2}$. Note that $h(n)\geq C\sqrt{\log n}$.
When bounding $\mathbb{P}\left( z(l) \in \An ,\ X(n) = l+\frac{n}{2} \right)$, let us split this probability into two parts, claiming that if $z(l)$ belongs to the aggregate, either a thin tentacle of settled particles branches out to that point, or there are many particles settled in a ball around $z(l)$. We will once again be using Lemma 2 of \cite{jerison2012logarithmic}, which was previously used in Section \ref{section: global upper bound}.
We write
\begin{multline}
	\mathbb{P}\left(z(l) \in \An,\ X(n) = z_1(l) = l+\frac{n}{2} \right) \\
	\leq 	\mathbb{P}\left(\#\left(\An\cap \comment{\mathbb{B}(z(l),l)}\right)>\beta l^d,\ X(n) = z_1(l) = l+\frac{n}{2} \right)\\
	+\mathbb{P}\left(z(l)\in \An,\ \#\left(\An\cap \mathbb{B}(z(l),l)\right)\leq \beta l^d \right),	\label{eqn: tentaclebis}
\end{multline}
where $\beta$ is chosen as in Lemma 2 of \cite{jerison2012logarithmic}, allowing us to handle the last term of \eqref{eqn: tentaclebis}. This gives
\begin{equation}
	\label{eqn: tentacle no sum}
	\mathbb{P}\left(z(l)\in \An,\ \#\left(\An\cap \mathbb{B}(z(l),l)\right)\leq \beta l^d \right)\leq K_0e^{-cl^2}.
\end{equation}
We write 
\begin{equation*}
	\sum_{l\geq C\sqrt{\log n }}K_0e^{-cl^2} \leq e^{-\frac{cC^2\log n}{2}}\sum_{l\geq C\sqrt{\log n}}e^{-\frac{cl^2}{2}}
	%& \leq e^{-\frac{cC^2\log n}{2}}\sum_{l \geq 0}e^{-\frac{cl^2}{2}} \\
	\leq Ke^{-\frac{cC^2\log n}{2}},
\end{equation*}
which, \comment{by taking $C$ large enough}, is smaller than any power of $n^{-1}$.
%, since $h(n)\geq C\sqrt{\log n}$. 
%TODO NEW
\comment{Let us now define a new random aggregate, namely $\tilde{A}$, built using the following protocol. For an initial configuration $\eta$, we throw particles according to the usual IDLA protocol, but we freeze all particles once they reach $\partial \mathcal{R}_{n/2}$. Once all particles have been thrown from $\eta$, we unfreeze the particles stationed on $\partial \mathcal{R}_{n/2}$ and let them continue their trajectory (until they exit the aggregate).
	\newline
	From the Abelian property, we know that $\tilde{A}(\eta) \overset{\mathrm{law}}{=} A(\eta)$. In particular, we have that 
	\begin{equation}
		\label{eqn: cardinal A tildeA}
		\#\left(\An\cap \mathbb{B}(z(l),l)\right) \overset{\mathrm{law}}{=} \#\left( \tilde{A}_n[\infty] \cap \mathbb{B}(z(l),l)\right),
	\end{equation}
	where $\tilde{A}_n[\infty]$ is obtained by taking the increasing union over $M\geq 0$ of $\tilde{A}_n[M]$.
	To handle the first term of the right-hand side of \eqref{eqn: tentaclebis}, we follow the same method as in \cite{chenavier2023bi}.}
%TODO END NEW
For an initial configuration $\eta$ and a set $B\subset \mathbb{Z}^d$, \comment{with respect to the aggregate $\tilde{A}(\eta)$}, we denote by $M^*_{n/2+l}\left(\eta, B\right)$ the number of random walks with initial configuration $\eta$ and which satisfy the following conditions:
\begin{itemize}
	\item The random walks intersect $B$ before they exit  $\mathcal{R}_{n/2+l}$
	\item The particles associated with the random walks hit $\partial \mathcal{R}_{n/2}$ 
\end{itemize}
The path of such a random walk is illustrated by Figure \ref{fig: walk in Mn/2} below: the aggregate is contained inside the dashed blue lines, the particle's path is represented by the red line, and the associated random walk's path by the dashed red line. 
\newline 
Just as we did for $A_n[\infty]$, we can define the random variable 
\[
\tilde{X}(n)=\max\left\{|z_1|,\ z\in \tilde{A}_n[\infty]\cap \mathbb{Z}_{n^{\alpha}}\right\}.
\]
Since $\tilde{A}_n[\infty] \overset{\mathrm{law}}{=} \An$, we have that $\tilde{X}(n) \overset{\mathrm{law}}{=} X(n)$. \newline
Notice that on the event ${\{\tilde{X}(n)=z_1(l)\}}$, the aggregate's furthermost point (on the $x$ axis) is therefore $z_1(l)$ and has a coordinate equal to $z_1(l)= l+\frac{n}{2}$. This implies the following inequality:
\begin{equation}
	\label{eqn: ineq M* tilde}
	\#\left(\comment{\tilde{A}_n[\infty]}\cap \mathbb{B}(z(l),l) \right) \overset{\mathrm{a.s.}}{\leq} M^*_{n/2+l}\left(\config, \mathbb{B}^{-}(z(l),l)\right),
\end{equation}
where $\mathbb{B}^{-}(z(l),l):=\mathbb{B}(z(l),l)\cap \mathcal{R}_{n/2+l}$.
Indeed, on the event $\tilde{X}(n)=z_1(l)$, the aggregate cannot go any further than $\mathcal{R}_{n/2+l}$ and therefore any particle of $\tilde{A}_n[\infty]\cap \mathbb{B}(z(l),l)$ necessarily hits $\partial \mathcal{R}_{n/2}$ before exiting the aggregate, and the random walk associated to that particle also necessarily intersected $\mathbb{B}^{-}(z(l),l)$ before exiting $\mathcal{R}_{n/2+l}$. The crucial point is that no particle can escape from $\mathcal{R}_{n/2+l}$ on this event.
This implies inequality \eqref{eqn: ineq M* tilde}. Combining this with \eqref{eqn: cardinal A tildeA}, we have that 
\begin{equation}
	\label{eqn: ineq M*}
	\#\left(\An\cap \mathbb{B}(z(l),l) \right) \overset{\mathrm{sto.}}{\leq} M^*_{n/2+l}\left(\config, \mathbb{B}^{-}(z(l),l)\right).
\end{equation}

\begin{figure}[!h]
	\centering
	\begin{tikzpicture}[scale=0.7, use Hobby shortcut,closed=true]
		\fill [color=gray!40] (6,0) .. (4,2) .. (2,3) .. (0,2) .. (-4,1.5) .. (-5.7, -2) .. (-4,-5) .. (0,-6) .. (5,-6) .. (5,-3) .. (5.3,-2) ;
		%	\draw (-6,-6) grid (6,3);
		\fill [color=gray!20] (6,-.3) circle (1) ;
		\draw [dashed, blue] (6,0) to[closed,curve through={(4,2) .. (2,3) .. (0,2) .. (-4,1.5) .. (-5.7, -2) .. (-4,-5) .. (0,-6) .. (5,-6) .. (5,-3) .. (5.3,-2)}] (6,0) ;
		\draw [->,very thick] (0,-6) -- (0,2);
		\draw [very thick](-6,-7) -- (-6,2);
		\draw [very thick](6,-7) -- (6,2);
		\draw [->,very thick] (-7,0) -- (7,0);
		\draw [red] (0,-2)to[curve through={(-2,1)..(3,-4) .. (4,-5) .. (3,-6) .. (4,-2) .. (4.5,-2.2) .. (5,-1.7)  }](5.5,-1); 
		\draw [red, dashed] (5.5,-1)to[curve through={(6,-0.5)}](8,-0.5); 
		\draw [dashed]  (5,-6) -- (5,2);	
		\draw [dashed]  (-5,-6) -- (-5,2);
		\draw (6,-.3) node {$\times$};
		\draw (6,-.3) node[right] {$z(l)$};
		\draw (-6,-7) node[below] {$\partial \mathcal{R}_{n/2+l}$};
		\draw (6,-7) node[below] {$\partial \mathcal{R}_{n/2+l}$};
		\draw (0,-6) node[below] {$\mathcal{H}$};
		\draw (5,-6) node[below] {$\partial \mathcal{R}_{n/2}$};
		\draw (-5,-6) node[below] {$\partial \mathcal{R}_{n/2}$};
		\filldraw [red] (5.58,-0.89) circle (1pt);
	\end{tikzpicture}
	\caption{An example of a walk in $M^*_{n/2+l}\left(\config,\mathbb{B}(z(l),l)\right)$
	}
	\label{fig: walk in Mn/2}
\end{figure}

Therefore, 
\begin{multline*}
	\mathbb{P}\left(\#\left(\An\cap \mathbb{B}(z(l),l)\right)>\beta l^d,\ z_1(l)=X(n) \right)\\
	\leq \mathbb{P}\left( M^*_{n/2+l}\left(\config, \mathbb{B}^{-}(z(l),l)\right) > \beta l^d\right).
\end{multline*}
The trajectories of the walks counted by $M^*$ are not independent. This comes from the fact that the \emph{particles} with which they are associated are killed upon exiting the current aggregate. Since the trajectories of the particles are highly dependent, this remains true for $M^*$. However, we only count walks for which the associated particle has hit $\partial \mathcal{R}_{n/2}$ before exiting the aggregate, meaning that the random walks evolve \emph{independently} after hitting $\partial \mathcal{R}_{n/2}$. In particular, they evolve independently once they reach $\mathbb{B}^{-}(z(l),l)$. We know that the walks counted my $M^*$ are independent after they reach $\partial \mathcal{R}_{n/2}$, so in particular after reaching $\partial \mathbb{B}^{-}(z(l),l)$. Recall that a random walk started in $x\in \partial \mathbb{B}^{-}(z(l),l)$ has probability at least $\rho>0$ to hit \comment{$\mathbb{B}(z(l),2l)^c \cap \mathcal{R}_{n/2+l}^c$} when it exits $\mathbb{B}(z(l),2l)$. It is important to note that $\rho$ does not depend on $n$ or $l$. In particular, random walks counted in $M^*_{n/2+l}$ have probability at least $\rho$ of hitting tile $\tilde{\tau}(z(l))$, where 
\[
\tilde{\tau}(z(l))=\mathbb{B}(z(l),2l)\cap \partial \mathcal{R}_{n/2+l}.
\]
We now split our probability into two, conditioning on the event where a sufficient amount of random walks of $M^*_{n/2+h(n)}$ have hit the tile $\tilde{\tau}(z(l))$ defined above. This gives 
\begin{multline}
	\mathbb{P}\left( M^*_{n/2+l}\left(\config, \mathbb{B}^{-}(z(l),l)\right)>\beta l^d \right)\\
	\leq 	\mathbb{P}\left(M^*_{n/2+l}\left(\config, \tilde{\tau}(z(l))\right)\leq\frac{\beta l^d\rho}{2} \;\middle|\; M^*_{n/2+l}\left(\config, \mathbb{B}^{-}(z(l),l)\right)>\beta l^d \right)\\
	+\mathbb{P}\left(M^*_{n/2+l}\left(\config, \tilde{\tau}(z(l))\right)>\frac{\beta l^d\rho}{2}\right).\label{eqn: binom bis}
\end{multline}

From what we said above, we know that the random walks are independent after they hit $\partial \mathbb{B}^{-}(z(l),h(n))$, and that any walk started from a point $x\in\partial \mathbb{B}^{-}(z(l),l)$ has probability at least $\rho$ to hit $\tilde{\tau}(z(l))$.\newline Therefore, conditional on the event $\Bigl\{M^*_{n/2+l}\left(\config, \mathbb{B}^{-}(z(l),l)\right)>\beta l^d\Bigr\}$, the random variable  $M^*_{n/2+l}\left(\config, \tilde{\tau}(z(l))\right)$ stochastically dominates a binomial distribution with parameters $\beta l^d$ and $\rho$, denoted by $\mathcal{B}\left(\beta l^d,\rho\right)$. 
This allows us to handle the first term in \eqref{eqn: binom bis}.
Indeed, using the fact that $\mathbb{E}\left[\mathcal{B}\left(\beta l^d,\rho\right)\right] = \beta l^d\rho$, standard concentration inequality theory (e.g \cite{BoucheronS.2013CiAn}) gives: 
\[
\mathbb{P}\left(\mathcal{B}\left(\beta l^d,\rho\right)\leq \frac{\beta l^d \rho}{2}\right)\leq \exp\left(-\frac{\beta l^d\rho}{8}\right).
\]
Therefore, using the stochastic domination we mentioned above,
\begin{multline*}
	\mathbb{P}\left(M^*_{n/2+l}\left(\config, z(l),\tilde{\tau}(z(l))\right)\leq\frac{\beta l^d\rho}{2} \;\middle|\; M^*_{n/2+l}\left(\config, \mathbb{B}^{-}(z(l),l)\right)>\beta l^d \right)\\
	\leq\exp\left(-\frac{\beta l^d\rho}{8}\right).
\end{multline*}
Proceeding as with \eqref{eqn: tentacle no sum}, we can show that
\[
\sum_{l \geq C\sqrt{\log n}}\exp\left(-\frac{\beta l^d\rho}{8}\right) \leq K \exp\left( - \frac{\beta C^d {\log (n)}^{d/2}}{16} \right),
\]
which, for $C$ taken large enough, decreases faster than any power of $n^{-1}$.

%This last term decreases faster than any power of $n^{-1}$.

Switching our focus to the second term \eqref{eqn: binom bis}, it remains to prove that for all $L>0$ and $n$ large,
\[
\sum_{l\geq C\sqrt{\log n}}\mathbb{P}\left(M^*_{n/2+l}\left(\config, \tilde{\tau}(z)\right)>\frac{\beta l^d\rho}{2}\right)\leq n^{-L}.
\]
%It is sufficient to show this for $\Mnstar(\init,\tilde{\tau}(z))$ instead of $\Mnstar(\config,\tilde{\tau}(z))$ simply using the fact that $\Mnstar(\init,\tilde{\tau}(z))\geq \Mnstar(\config,\tilde{\tau}(z)) $.
We will know be using a lemma similar to Lemma \ref{lemma: 6.2}, which we give below.
\begin{lemma} (Lemma 2.5 of \cite{asselah2013sublogarithmic})
	\label{lemma AG Cheb}
	Suppose that a sequence of random variables \newline $W_n,\ L_n,\ M_n,\ \widetilde{M}_n$ and an event $\mathcal{A}_n$ satisfy for each $n\in \N$: 
	\[
	(W_n+L_n)\mathbbm{1}_{\mathcal{A}_n}\underset{sto}{\leq}\widetilde{M}_n \quad \mathrm{ and } \quad M_n\overset{\mathrm{law}}{=}\widetilde{M}_n.
	\]
	Assume that $W_n$ and $L_n$ are independent, and that $L_n$ and $M_n$ are series of independent Bernoulli random variables with finite expectations, with $L_n=\sum_{i\geq 0}Y_i^{(n)}$.
	Finally, assume that $\mu_n=\mathbb{E}\left[M_n\right]-\mathbb{E}\left[L_n\right]\geq 0$.
	\newline
	Then, for all $n\geq 0$, $\xi_n\in \R$ and $\lambda\in [0,\log 2]$, 
	\[
	\mathbb{P}\left(W_n\geq \xi_n, \mathcal{A}_n\right) \leq \exp\left( -\lambda(\xi_n-\mu_n)+\lambda^2\left(\mu_n + 4\sum_{i\geq 0}\mathbb{E}{Y_i^{(n)}}^2   \right)  \right).
	\]
\end{lemma}
Using the same arguments as for \eqref{eqn: stoch eq}, we can establish the following equality:

\begin{equation}
	\label{eqn: stoch eq bis}
	\Mnstarl\left(\init, \tilde{\tau}(z(l))\right) + \Mnl\left(A_{n/2}(\init),\tilde{\tau}(z(l))\right) \overset{\mathrm{law}}{=} \tilde{M}_n,
\end{equation}
where ${A_{n/2}(\init)=A(\init)\cap \mathcal{R}_{n/2}}$, and $\tilde{M}_n$ is an independent copy of \newline $\Mnl\left(\init, \tilde{\tau}(z(l))\right)$. 

The idea here is to once again consider a walk counted by $\Mnl\left(\init, \tilde{\tau}(z(l))\right)$, and consider the trajectory of its associated particle. Either the particle has hit $\partial \mathcal{R}_{n/2}$ before exiting the aggregate, and is therefore counted by $\Mnstarl\left(\init,\tilde{\tau}(z(l))\right)$, or the particle has settled before on some site $x$. In that case, we can launch a new random walk from $x\in A_{n/2}(\init)$, which is accounted for by $\Mnl\left(A_{n/2}(\init),\tilde{\tau}(z(l))\right) $. 
Now, take $\alpha'>\alpha$. Let us denote by $\delta_I(n)$ the inner error of $A_{n/2}(\init)$ on $\mathbb{Z}_{n^{\alpha'}}$, that is
\[
\delta_I(n)=\max\left\{\frac{n}{2}-|z_1|,\ z\in\left(\mathcal{R}_{n/2}\setminus A_{n/2}(\init)\right)\cap\mathbb{Z}_{n^{\alpha'}}\right\}.
\]
This quantity is illustrated in Figure \ref{fig: inner error}. 
\begin{figure}[h]
	\centering
	\begin{tikzpicture}[scale=0.6, use Hobby shortcut,closed=true]
		\fill [color=gray!40] (4.6,2.9)..(4.5,6) .. (0,6).. (-4.5,4.5).. (-4,3.5).. (-4.8,-1).. (-5,-2).. (-6,-2.5).. (-6,-3).. (-2.4,-6).. (0,-6.7).. (4.5,-4) .. (4.5,0) ..(4.6,2.9) ;
		\fill [color=white] (-8,-2) rectangle (-5,-5);
		\draw [->,very thick] (0,-6) -- (0,6);
		\draw [thick]  (5,-6) -- (5,6);	
		\draw [thick]  (-5,-6) -- (-5,6);
		\draw [dashed,blue](-6,5) -- (6,5);
		\draw [dashed,blue](-6,-5) -- (6,-5);
		\draw [<->](-5,2.7) -- (-3.82,2.7);
		\draw (-6,5) node [left] {$\mathbb{Z}_{n^{\alpha'}}$};
		\draw (-6,-5) node [left] {$\mathbb{Z}_{n^{\alpha'}}$};
		\draw (-4.35,2.7) node[above] {$\delta_I(n)$};
		\draw (0,-6) node[below] {$\mathcal{H}$};
		\draw (5,-6) node[below] {$\partial \mathcal{R}_{n/2}$};
		\draw (-5,-6) node[below] {$\partial \mathcal{R}_{n/2}$};
		%\draw (-5,6 )grid (5,-6);
	\end{tikzpicture}
	\caption{Illustration of the inner error}
	\label{fig: inner error}
\end{figure}

We will be using the work we did in the previous section to bound this inner error. 
%Indeed, we know that $\mathbb{P}\left(\delta_I(n)> \kappa\sqrt{\log n}\right)$ is smaller than any power of $n^{-1}$, provided that $\kappa$ and $n$ are chosen large enough.
Using the definition of $\delta_I(n)$, we have $\mathcal{R}_{n/2-\delta_I(n)}\cap \mathbb{Z}_{n^{\alpha'}}\subset A_{n/2}(\init)$, so combining this with \eqref{eqn: stoch eq bis}, we get
\[
\Mnstarl\left(\init, \tilde{\tau}(z(l))\right) + \Mnl\left(\mathcal{R}_{n/2-\delta_I(n)}\cap \mathbb{Z}_{n^{\alpha'}},\tilde{\tau}(z(l))\right) \overset{\mathrm{sto.}}{\leq} \tilde{M}_n.
\]
Now, for some $\alpha_d>0$ which will be chosen later, on the event \commentK{${\left\{\delta_I(n)\leq \frac{\alpha_dl}{2C}\right\}}$}, we have $\commentK{\mathcal{R}_{n/2-  \frac{\alpha_d l}{2C}}}\cap \mathbb{Z}_{n^{\alpha'}} \subset \mathcal{R}_{n/2-\delta_I(n)}\cap \mathbb{Z}_{n^{\alpha'}}$, \commentK{with the convention that $\mathcal{R}_m = \mathbbm{1}_{\mathcal{H}}$ when $m \leq 0$.} This gives:
\begin{equation*}
	\left(\Mnstarl\left(\init, \tilde{\tau}(z(l))\right) + \Mnl\left(\commentK{\mathcal{R}_{n/2-  \frac{\alpha_d l}{2C}}}\cap \mathbb{Z}_{n^{\alpha'}},\tilde{\tau}(z)\right)\right) \commentK{\mathbbm{1}_{\delta_I(n)\leq \frac{\alpha_dl}{2C}}}
	\overset{\mathrm{sto.}}{\leq} \tilde{M}_n.
\end{equation*}
This stochastic inequality is similar to the one required by Lemma \ref{lemma AG Cheb}. We can directly apply this lemma with 
\[
\begin{cases}
	W_n=\Mnstarl\left(\init, \tilde{\tau}(z(l))\right) \\
	L_n=\Mnl\left(\commentK{\mathcal{R}_{n/2-  \frac{\alpha_d l}{2C}}}\cap \mathbb{Z}_{n^{\alpha'}},\tilde{\tau}(z(l))\right) \\
	\widetilde{M}_n=\Mnl\left(\init,\tilde{\tau}(z(l))\right)\\
	\commentK{\mathcal{A}_n=\left\{\delta_I(n)\leq \frac{\alpha_d l}{2C}  \right\}}\\
	\xi_n=\frac{\beta l^d\rho}{2}
\end{cases}
\]
\newline
We show in Section \ref{sec: shape thm proofs} that the hypotheses in Lemma \ref{lemma AG Cheb} hold. 
\label{subsubsection: application of lemma bis}
Now, we write 
\begin{multline*}
	\mathbb{P}\left(M^*_{n/2+l}\left(\init, \tilde{\tau}(z(l))\right)>\frac{\beta l^d\rho}{2}\right) \\
	\leq\mathbb{P}\left(M^*_{n/2+l}\left(\init, \tilde{\tau}(z(l))\right)>\frac{\beta l^d\rho}{2},\ \delta_I(n)\leq \commentK{\frac{\alpha_d l}{2C}} \right)
	+\mathbb{P}\left(\delta_I(n)> \commentK{\frac{\alpha_d l}{2C}}\right).
\end{multline*}
We first focus our attention on the second term. We will be using our computations from the proof of the lower bound to control this term. \commentK{We have:}
\begin{align*}
	\sum_{l \geq C\sqrt{\log n}} \mathbb{P}\left(\delta_I(n)> \frac{\alpha_d l}{2C}\right) &\leq \sum_{l \geq \sqrt{\log n}} \mathbb{P}\left( \delta_I(n) > \frac{\alpha_d l}{2} \right) \\
	& \leq \sum_{l \geq \sqrt{\log n}} \exp(- \kappa \alpha_d l^2) \\
	& \leq K\exp\left( -\frac{\kappa \alpha_d \log n}{2} \right).
\end{align*}
Taking $\alpha_d$ sufficiently large, this term becomes smaller than any given power of $n^{-1}$.

%Notice that since $h(n)\geq C\sqrt{\log n}$, we have 
%\[
%\mathbb{P}\left(\delta_I(n)> \frac{\alpha_dh(n)}{2C}\right)\leq \mathbb{P}\left(\delta_I(n)> \frac{\alpha_d\sqrt{\log n}}{2}\right).
%\]
%Taking $\alpha_d$ sufficiently large, this term becomes smaller than any given power of $n^{-1}$.
We now shift our focus to the first term of the sum. Fix $\alpha_d$ large enough so that the previous term 
%$\mathbb{P}(\delta_I(n)> \frac{\alpha_dh(n)}{2C})$ 
is smaller than any power of $n^{-1}$.

After an application of Lemma \ref{lemma AG Cheb} and an optimization detailed in Section \ref{sec: upper bound proofs}, we get for some constant $\kappa>0$, if $C$ is chosen large enough:
\begin{equation}
	\mathbb{P}\left( M^*_{n/2+l}\left(\init, \mathbb{B}(z(l),\tilde{\tau}(z(l)))\right)>\frac{\beta l^d\rho}{2},\ \delta_I(n)\leq \frac{\alpha_d l}{2C}\right)
	\leq\exp\left(-\kappa l^2\right) \label{line: optimizationbis}.
\end{equation}
Hence, 
\begin{align*}
	\sum_{l \geq C\sqrt{\log n}} \mathbb{P}\left( M^*_{n/2+l}\left(\init, \mathbb{B}(z(l),\tilde{\tau}(z(l)))\right)>\frac{\beta l^d\rho}{2},\ \delta_I(n)\leq \frac{\alpha_d l}{2C}\right) &\leq \sum_{l \geq C\sqrt{\log n}} \exp\left(-\kappa l^2\right) \\
	& \leq K\exp\left(-\frac{\kappa C^2 \log n}{2}\right),
\end{align*}
Which goes to zero faster than any power of $n^{-1}$, \comment{given $C$ is large enough}, and concludes the proof of our theorem.
\subsection{Auxiliary proofs}
\label{sec: shape thm proofs}
\subsubsection{Proofs from the lower bound}
\label{sec: lower bound proofs}
We start by giving the proof of \eqref{eqn: lower bound mu}. To do it, we apply the following lemma, which is a simple extension of Lemma 6.4 of \cite{chenavier2023bi}.
\begin{lemma}
	\label{lemma 6.4}
	Let $r\leq r'$ and let $\tau \subset\partial \mathcal{R}_{r'}$ be finite. Then 
	\[
	\mathbb{E}\left[{M_{r'}(\mathcal{R}_{r},\tau)}\right]=\frac{2r+1}{2}\#\tau.
	\]
	In particular, for any $r'\geq 1$,
	\begin{equation}
		\label{eq:expectationbis}
		\mathbb{E}\left[{M_{r'}(\mathbbm{1}_{\mathcal{H}}, \tau)}\right] = \frac{\#\tau}{2}.
	\end{equation}
\end{lemma}
Now, to get \eqref{eqn: lower bound mu}, we write
\begin{align*}
	\mu(\tau)&=	\mathbb{E}\left[\Mk(n\mathbbm{1}_\mathcal{H},\tau)\right] - \mathbb{E}\left[\Mk\left(\Rk\setminus\mathcal{Z},\tau\right)\right]\\
	&\geq n\mathbb{E}\left[\Mk(\mathbbm{1}_\mathcal{H},\tau)\right] - \mathbb{E}\left[\Mk\left(\Rk,\tau\right)\right]\\
	&\geq \frac{\#\tau}{2}\left(n-2k\sqrt{\log n}-1\right)\\
	&\geq cC\left(\sqrt{\log n}\right)^d.
\end{align*}
We now show the bound given by \eqref{eqn: E Yi}. \comment{Recall that $k$ is taken such that $k\leq \frac{n}{2\sqrt{\log n}}-C$.}
Using Lemma \ref{lemma 6.3}, we have
\begin{align*}
	\mathbb{P}_y\left(S(H_{k\sqrt{\log n}})\in\tau\right) &\leq \#\tau \max_{x\in\tau}\mathbb{P}_y\left(S(H_{k\sqrt{\log n}})=x\right)\\
	&\leq \#\tau \max_{x\in\tau} \frac{\kappa}{\|y-x\|^{d-1}}\\	
	&\leq \#\tau \frac{c}{\|y-z\|^{d-1}}.
\end{align*}
Hence
\[
\sum_{y\in\Rk \setminus\mathcal{Z}}\mathbb{P}_y\left(S(H_{k\sqrt{\log n}})\in\tau\right)^2 \leq (c\#\tau)^2 \sum_{y\in\Rk \setminus\mathcal{Z}}\frac{1}{\left(\|y-z\|^2\right)^{d-1}}.
\]
Since $\#\tau$ is of order $(\sqrt{\log n})^{d-1}$, it suffices to show that \comment{$\sum_{y\in\Rk \setminus\mathcal{Z}}\frac{1}{\left(\|y-z\|^2\right)^{d-1}}$ is bounded by a finite constant}, and therefore that
\[
\sum_{j=1}^{2k\sqrt{\log n}}\int_{{\left[ 1,\infty \right[}^{d-1} } \frac{\dx x_1\dots \dx x_{d-1}}{\left(j^2+x_1^2+\dots x_{d-1}^2\right)^{d-1}}
\]
is as well.
This is true since
\begingroup
\allowdisplaybreaks
\begin{align}
	\sum_{j=1}^{2k\sqrt{\log n}}\int_{{\left[ 1,\infty \right[}^{d-1} } \frac{\dx x_1\dots \dx x_{d-1}}{\left(j^2+x_1^2+\dots x_{d-1}^2\right)^{d-1}}
	&\leq \comment{c_d}\sum_{j=1}^{2k\sqrt{\log n}}\int_{1}^{\infty}\frac{r^{d-2}}{\left(j^2+r^2\right)^{d-1}}\, \dx r \nonumber \\ 
	&\leq \comment{c_d}\sum_{j=1}^{\comment{\infty}} \int_{1/j}^{\infty}\frac{j^{d-2}r^{d-2}}{j^{2(d-1)}\left(1+r^2\right)^{d-1}} j\dx r \nonumber \\ 
	&\leq \comment{c_d}\sum_{j=1}^{\comment{\infty}} \frac{1}{j^{d-1}} \int_{0}^{\infty}\frac{r^{d-2}}{\left(1+r^2\right)^{d-1}} \dx r, \label{eqn: integration}
\end{align}
\endgroup
\comment{where $c_d$ denotes the volume of the $(d-2)$-dimensional sphere.} Now, since $d\geq 3$, we have that \eqref{eqn: integration} is finite. Thus, for some $c>0$,
\begin{equation*}
	\sum_{y\in\Rk \setminus\mathcal{Z}}\mathbb{P}_y\left(S(H_{k\sqrt{\log n}})\in\tau\right)^2\leq c\log(n)^{d-1}.
\end{equation*}

\paragraph{Optimization in \eqref{eqn: optimization}:}
In this section, we detail the computations of the optimization in \eqref{eqn: optimization}. In all that follows, $\comment{\kappa'}$ denotes a generic constant.
\newline
We wish to minimize $\lambda\mapsto \exp\left( -\frac{\lambda}{3}\mu(\tau) +\frac{\lambda^2}{2}\left(\mu(\tau)+\comment{\eta}\log (n)^{d-1}\right) \right)$ on $\R_+$.
Recall \comment{from \eqref{eqn: lower bound mu}} that $\mu(\tau)\geq \kappa C\log(n)^{d/2}$. Note that it suffices to minimize the function within the exponential, which happens to be a second degree polynomial. Pick $\lambda^*$ minimizing the polynomial, that is 
\[
\lambda^*=\dfrac{\mu(\tau)}{3\left( \mu(\tau) +\comment{\eta}\log(n)^{d-1} \right)}\comment{.}
\]
This gives
\begin{align*}
	-\frac{\lambda^*}{3}\mu(\tau) +\frac{(\lambda^*)^2}{2}\left(\mu(\tau)+\comment{\eta}\log (n)^{d-1}\right)&\leq -\frac{\mu(\tau)^2}{18\left(\mu(\tau)+\comment{\eta} \log (n)^{d-1}\right)}\\
	&\leq -\frac{\mu(\tau)}{18\left(1+\frac{\comment{\eta} \log (n)^{d-1}}{\mu(\tau)}\right)}.
\end{align*}
Now, 
\[
1+\frac{\comment{\eta} \log (n)^{d-1}}{\mu(\tau)}\leq 1+\frac{\comment{\eta} \log (n)^{d-1}}{\kappa C\log(n)^{d/2}}\leq \frac{\comment{\kappa'}}{C}\log(n)^{d/2-1}.
\]
\comment{Therefore, combining this with} the fact that $\mu(\tau)\geq \kappa C\log(n)^{d/2}$, we get
\begin{align*}
	-\frac{\mu(\tau)}{18\left(1+\frac{\comment{\eta}\log (n)^{d-1}}{\mu(\tau)}\right)}&\leq -\frac{C\mu(\tau)}{\comment{\kappa'} \log(n)^{d/2-1}}\\
	&\leq -\frac{\kappa C^2\log(n)^{d/2}}{\comment{\kappa'}\log(n)^{d/2-1}}\\
	&\leq -\comment{\kappa'} C^2\log n.
\end{align*}

\subsubsection{Proofs from the upper bound}
\label{sec: upper bound proofs}
We now show proofs concerning the results for the upper bound of Theorem \ref{thm: shape}. Let us begin by showing that the hypotheses in Lemma \ref{lemma AG Cheb} hold. \commentK{In all that follows, we fix $l \geq C\sqrt{\log n}$. Recall that $z(l)$ is such that ${z(l)=(\frac{n}{2} +l, 0, \dots, 0)}$.}

\paragraph{Hypotheses of Lemma \ref{lemma AG Cheb}:}
We first need to check that \mbox{$\mu_n=\mathbb{E}\left[M_n\right]-\mathbb{E}\left[L_n\right]\geq 0$}. This once again uses Lemma \ref{lemma 6.4}.
To lighten notation, we define ${\mathcal{R}(l,n,\alpha_d,C):=\mathcal{R}_{n/2-  \frac{\alpha_d l}{2C}}}$. 
We have:
\begin{align*}
	\mu_n&=\mathbb{E}\left[\Mnl\left(\init,\tilde{\tau}(z(l))\right)\right]-\mathbb{E}\left[\Mnl\left(\mathcal{R}(l, n,\alpha_d,C)\cap \mathbb{Z}_{n^{\alpha'}},\tilde{\tau}(z(l))\right)\right]\\
	&\geq \mathbb{E}\left[\Mnl\left(\init,\tilde{\tau}(z(l))\right)\right]-\mathbb{E}\left[\Mnl\left(\mathcal{R}(l, n,\alpha_d,C),\tilde{\tau}(z(l))\right)\right]\\
	&= \frac{\#\tilde{\tau}(z(l))}{2}\commentK{\left(n - 2\left(n/2 - \min\left( \frac{\alpha_d l}{2C},\ n/2  \right) \right) -1 \right)}\\
	&= \frac{\#\tilde{\tau}(z(l))}{2}\commentK{\left(\min\left(\frac{\alpha_d l}{C},\ n\right) -1 \right)}.
	%	&\commentK{ \geq \frac{c\alpha_d}{2C}l^d.}
\end{align*}
In the case where $\frac{\alpha_d l}{C} \geq n$, then the minimum is $n$, and so $\mu_n \geq cn\#\tilde{\tau}(z(l))$ for some positive constant $c$. Since $\#\tilde{\tau}(z(l))$ is of order $l^{d-1}$ we have that $\mu_n \geq cl^{d-1}$.

In the other case, the minimum is $\frac{\alpha_d l}{C}$, and so in this case $\mu_n \geq c\#\tilde{\tau}(z(l))\frac{\alpha_d l}{C} = \frac{c\alpha_d}{C}l^d$, for some positive constant $c$. In both cases, the condition $\mu_n \geq 0$ is verified.

Now, following the same computations as the ones used to get \eqref{eqn: E Yi}, we get
\begin{align*}
	\sum_{y\in \mathcal{R}(l, n,\alpha_d,C)\cap \mathbb{Z}_{n^{\alpha'}} }\mathbb{P}_y\left(S(H_{n/2+l})\in\tilde{\tau}(z(l))\right)^2&\leq \sum_{y\in \mathcal{R}(l, n,\alpha_d,C) }\mathbb{P}_y\left(S(H_{n/2+l})\in\tilde{\tau}(z(l))\right)^2 \\
	&\commentK{\leq cl^{2d-2}.}
\end{align*}

\paragraph{Control of $\mu_n$:}
\label{subsubsection: control of mu}
We showed just above that the hypotheses of Lemma \ref{lemma AG Cheb} were satisfied. However, when applying this lemma and optimizing, we need an upper bound on $\mu_n$, given by the following proposition. 
\begin{lemma}
	\label{lemma: upper bound mu}
	There exist positive constants $c,\ C_0$ such that for $n$ large enough, for all $l > C\sqrt{\log n}$,
	\[
	\mu_n\leq \frac{c\alpha_d \commentK{l^d}}{C}+C_0.
	\]
\end{lemma}
\begin{proof}[Proof of Lemma \ref{lemma: upper bound mu}:]
	Here, we must once again consider two cases. When $n \leq \frac{\alpha_d l}{C}$, then we write 
	\begin{align*}
		\mu_n &= \mathbb{E}\left[\Mnl\left(\init,\tilde{\tau}(z(l))\right)\right]-\mathbb{E}\left[\Mnl\left(\mathcal{R}(l, n,\alpha_d,C)\cap \mathbb{Z}_{n^{\alpha'}},\tilde{\tau}(z(l))\right)\right]\\
		& \leq \mathbb{E}\left[\Mnl\left(\init,\tilde{\tau}(z(l))\right)\right] \\
		& =n\frac{\#\tilde{\tau}(z(l))}{2}\\
		& \leq cl^{d-1}n \\
		& \leq cl^{d-1}\cdot \frac{\alpha_d l}{C}\\
		& = \frac{c\alpha_d l^d}{C}
	\end{align*}
	In this case, $C_0 = 0$.
	
	Now, consider the other case where $n > \frac{\alpha_d l}{C}$. 
	We write
	\begin{align*}
		\mu_n&=\mathbb{E}\left[\Mnl\left(\init,\tilde{\tau}(z(l))\right)\right]-\mathbb{E}\left[\Mnl\left(\mathcal{R}(l, n,\alpha_d,C),\tilde{\tau}(z(l))\right)\right]\\
		&+\mathbb{E}\left[\Mnl\left(\mathcal{R}(l, n,\alpha_d,C),\tilde{\tau}(z(l))\right)\right]-\mathbb{E}\left[\Mnl\left(\mathcal{R}(l, n,\alpha_d,C)\cap \mathbb{Z}_{n^{\alpha'}},\tilde{\tau}(z(l))\right)\right].
	\end{align*}
	The first line has already been computed above, and can be bounded from above by \commentK{$\frac{c\alpha_d}{C}l^d$}, for some $c>0$. We now shift our focus on the second line. Notice that this quantity is equal to:
	\begin{equation}
		\label{eqn: avg numb walks}
		\mathbb{E}\left[\Mnl\left(\mathcal{R}(l, n,\alpha_d,C)\cap \mathbb{Z}_{n^{\alpha'}}^c,\tilde{\tau}(z(l))\right)\right],
	\end{equation}
	and it remains to show that this is bounded uniformly on $n$.
	\newline
	Now, recall that $z(l)\in\mathbb{Z}_{n^{\alpha}}$. Walks counted by $\Mnl\left(\mathcal{R}(l, n,\alpha_d,C)\cap \mathbb{Z}_{n^{\alpha'}}^c,\tilde{\tau}(z(l))\right)$ necessarily stay within $\mathcal{R}_{n/2+l}$ between levels $n^{\alpha'}$ and $n^{\alpha}$ before exiting through $\tilde{\tau}(z(l))$. We can therefore use a donut argument (just like in the proof of Theorem \ref{thm: strong stab}) between levels $n^{\alpha}$ and levels $n^{\alpha'}$ and beyond, by building hypercubes of length $n+2l$ between these levels. The length of our hypercubes is imposed to us by the width of $\mathcal{R}_{n/2+l}$, which equals $n+2l$. Since we want the walk to stay inside the slab, we choose our boxes to have the same width.
	We illustrate this argument with Figure \ref{fig: boxmethod}.
	\begin{figure}[h]
		\centering
		\begin{tikzpicture}[scale=0.6]
			\fill [color=gray!40](-1,2) rectangle (1,14);
			\draw [->,very thick] (0,-1) -- (0,14.3);
			\draw [->,very thick] (-2,0) -- (2,0);
			\draw [dashed, blue] (-2,2) -- (2,2);
			\draw (-2,2) node[left, blue] {$n^{\alpha}$};
			\draw [dashed, red] (-2,13) -- (1.25,13);
			\draw (-2,13) node[left, red] {$n^{\alpha'}$};
			\draw (-1,0) node {|};
			\draw (1,0) node {|};
			\draw (1.5,-1.25) node[below] {$\partial \mathcal{R}_{n/2+l}$};
			\draw [dashed](-1,-1)--(-1,14.5);
			\draw [dashed](1,-1)--(1,14.5);
			\draw [black,thick](-1,2) rectangle (1,4);
			\draw [black,thick](-1,4) rectangle (1,6);
			\draw [black,thick](-1,6) rectangle (1,8);
			\draw [black,thick](-1,8) rectangle (1,10);
			\draw [black,thick](-1,10) rectangle (1,12);
			\draw [black,thick](-1,12) rectangle (1,14);
			\draw [<->,thick] (-1,14.5) -- (1,14.5);
			\draw (0,14.5) node[above] {$n+2l$};
			\draw [<->,thick] (1.5,14) -- (1.5,12);
			\draw (1.5,13) node[right] {$n+2l$};
			\draw [red] (-0.3,13.4)to[curve through={(0.4,11).. (0.1,6) (2,1)}](2,1);
			\filldraw [red] (-0.3,13.4) circle (1pt);
			\draw (-0.3,13.4) node[left,red] {$x$};
		\end{tikzpicture}
		\caption{Illustration of the box method}
		\label{fig: boxmethod}
	\end{figure}
	Now, for a walk started on some site of $\mathcal{R}(l, n,\alpha_d,C)\cap \mathbb{Z}_{n^{\alpha'}}^c$ to exit $\mathcal{R}_{n/2+l}$ through $\tilde{\tau}(z(l))$, it necessarily has to cross a certain amount of boxes. Just like in Proposition \ref{prop: crossing prob}, we use the same reasoning to say that:
	\[
	\forall k\in \N,\ \mathbb{P}\left(\text{the walk goes through at least $k$ boxes}\right)\leq (1-c)^{k},
	\]
	with $c=\frac{1}{4d^2}$. We need to determine the minimum amount of cubes of length $n+2l$ that can fit between $n^{\alpha}$ and $n^{\alpha'}$. This is equal to	$\frac{n^{\alpha'}-n^{\alpha}}{n+2l}$, which is greater than $n$ given $n$ is large enough and $\alpha'>\max(\alpha,2)$. This uses the fact that $n + 2l $ is of order $n$. Let us number these boxes by $B_0,\ B_1, \dots$ starting from level $n^{\alpha}$. We have:
	\[
	\mathcal{R}(l, n,\alpha_d,C)\cap\mathbb{Z}_{n^{\alpha'}}^c\subset \bigcup_{k\geq n} B_k.
	\]
	Hence:
	\begin{multline*}
		\mathbb{E}\left[\Mnl\left(\mathcal{R}(l, n,\alpha_d,C)\cap \mathbb{Z}_{n^{\alpha'}}^c,\tilde{\tau}(z(l))\right)\right]\\
		\begin{split} & =\mathbb{E}\left[\Mnl\left(\bigcup_{k\geq n} B_k , \tilde{\tau}(z(l)) \right)\right] \\
			&\leq \sum_{k\geq n}\sum_{x\in B_k} \mathbb{P}_x\left(S(H_{n/2+l})\in\tilde{\tau}(z(l))\right) \\
			&\leq \sum_{k\geq n}\sum_{x\in B_k}\mathbb{P}_x\left(\text{the walk goes through at least $k-1$ boxes}\right)\\
			&\leq \sum_{k\geq n} (n+2l)^d (1-c)^{k-1}\\
			%			&\leq Kn^d(1-c)^{n-1}\sum_{k\geq n}(1-c)^{k-n}\\
			&\leq Kn^d(1-c)^{n-1},
		\end{split}	
	\end{multline*}
	which tends to 0 when $n$ tends to infinity, and is consequently uniformly bounded on $n$.
	
	In this case, we have that 
	\[
	\mu_n\leq \frac{c\alpha_d \commentK{l^d}}{C}+C_0,
	\]
	for some constant $C_0$ that does not depend on $n$.
\end{proof}
\paragraph{Optimization in \eqref{line: optimizationbis} :}
We end by detailing the optimization in \eqref{line: optimizationbis}. This computation follows the same spirit as the optimization in the lower bound (see \eqref{eqn: optimization}). We will be using the previous bound on $\mu_n$ given by Lemma \ref{lemma: upper bound mu} to conclude.

In all that follows, recall that $\alpha_d$ is fixed. \commentK{Let $\kappa$ denote a generic positive constant}. After applying Lemma \ref{lemma AG Cheb}, we get that for all $\lambda>0$:
\begin{multline*}
	\mathbb{P}\left( M^*_{n/2+l}\left(\init, \mathbb{B}(z(l),\tilde{\tau}(z(l)))\right)>\frac{\beta l^d\rho}{2},\ \delta_I(n)\leq \frac{\alpha_d l}{2C}\right) \\
	\leq \exp\left(-\lambda\left(\frac{\beta l^d\rho}{2}-\mu_n\right) + \lambda^2\left(\mu_n + 4\sum_{y\in \mathcal{R}(l, n,\alpha_d,C)\cap \mathbb{Z}_{n^{\alpha'}}} \mathbb{P}_y\left(S(H_{n/2+l})\in\tilde{\tau}(z(l))\right)^2 \right)\right).
\end{multline*}	
Once again, minimizing in $\lambda>0$ yields:
\begin{multline*}
	\mathbb{P}\left( M^*_{n/2+l}\left(\init, \mathbb{B}(z(l),\tilde{\tau}(z(l)))\right)>\frac{\beta l^d\rho}{2},\ \delta_I(n)\leq \frac{\alpha_d l}{2C}\right)\\ 
	\leq \exp\left(-\frac{\left(\mu_n-\frac{\beta l^d\rho}{2}\right)^2}{4\left(\mu_n+4\displaystyle{\sum_{y\in \mathcal{R}(l, n,\alpha_d,C)\cap \mathbb{Z}_{n^{\alpha'}}}}\mathbb{P}_y\left(S(H_{n/2+l})\in\tilde{\tau}(z(l))\right)^2 \right)}\right).
	%\leq \exp\left(-c\left(\frac{\alpha_d}{C}h(n)\right)^2\right)
\end{multline*}	
Using the bound of Lemma \ref{lemma: upper bound mu}, we can take $C$ sufficiently large in order for $\left(\mu_n-\frac{\beta l^d\rho}{2}\right)^2$ to be of order $l^{2d}$.
Therefore, we have for $C$ large enough, $(\frac{\beta l^d\rho}{2}-\mu_n)^2\geq \commentK{\kappa} l^{2d}$.
Now, recall that 
\[
\sum_{y\in \mathcal{R}(l, n,\alpha_d,C)\cap \mathbb{Z}_{n^{\alpha'}} }\mathbb{P}_y\left(S(H_{n/2+h(n)})\in\tilde{\tau}(z)\right)^2\leq \commentK{\kappa}l^{2d-2}.
\]
Using the bound of Lemma \ref{lemma: upper bound mu} and the fact that $d>2$, we have:
\[
4\left(\mu_n+4\sum_{y\in \mathcal{R}(l, n,\alpha_d,C)\cap \mathbb{Z}_{n^{\alpha'}}} \mathbb{P}_y\left(S(H_{n/2+l})\in\tilde{\tau}(z(l))\right)^2 \right)\leq \commentK{\kappa}l^{2d-2}.
\]
Combining both results gives:
\begin{align*}
	\mathbb{P}\left( M^*_{n/2+l}\left(\init, \mathbb{B}(z(l),\tilde{\tau}(z(l)))\right)>\frac{\beta l^d \rho}{2},\ \delta_I(n)\leq \frac{\alpha_d l}{2C}\right)&\leq  \exp\left(- \frac{\commentK{\kappa}l^{2d}}{\commentK{\kappa}l^{2d-2}}\right)\\
	&=\exp\left(-\commentK{\kappa} l^2\right).
\end{align*}

\bibliographystyle{acm}
\bibliography{article1_bib.bib}

\begin{thebibliography}{10}

\bibitem{asselah2013logarithmic}
{\sc {Asselah}, A., and {Gaudilli\`ere}, A.}
\newblock {From logarithmic to subdiffusive polynomial fluctuations for
  internal DLA and related growth models.}
\newblock {\em {The Annals of Probability.} 41}, 3A (2013), 1115--1159.

\bibitem{asselah2013sublogarithmic}
{\sc {Asselah}, A., and {Gaudilli\`ere}, A.}
\newblock {Sublogarithmic fluctuations for internal DLA.}
\newblock {\em {The Annals of Probability.} 41}, 3A (2013), 1160--1179.

\bibitem{AG14}
{\sc {Asselah}, A., and {Gaudilli\`ere}, A.}
\newblock {Lower bounds on fluctuations for internal DLA.}
\newblock {\em {Probab. Theory Relat. Fields} 158}, 1-2 (2014), 39--53.

\bibitem{asselah2019outer}
{\sc Asselah, A., and Gaudilli{\`e}re, A.}
\newblock On outer fluctuations for internal {DLA}.
\newblock {\em arXiv preprint arXiv:1904.10168\/} (2019).

\bibitem{AR16}
{\sc Asselah, A., and Rahmani, H.}
\newblock Fluctuations for internal {DLA} on the comb.
\newblock {\em Ann. Inst. Henri Poincar\'{e} Probab. Stat. 52}, 1 (2016),
  58--83.

\bibitem{BDCKL}
{\sc Benjamini, I., Duminil-Copin, H., Kozma, C., and {Lucas}, C.}
\newblock Internal diffusion-limited aggregation with uniform starting points.
\newblock {\em Ann. Inst. Henri Poincar\'{e} Probab. Stat. 56}, 1 (2020),
  391--404.

\bibitem{benjamini2001percolation}
{\sc Benjamini, I., and Schramm, O.}
\newblock Percolation in the hyperbolic plane.
\newblock {\em J. Amer. Math. Soc. 14}, 2 (2001), 487--507.

\bibitem{B04}
{\sc Blach\`ere, S.}
\newblock Internal diffusion limited aggregation on discrete groups of
  polynomial growth.
\newblock In {\em Random walks and geometry}. Walter de Gruyter, Berlin, 2004,
  pp.~377--391.

\bibitem{BlB07}
{\sc Blach\`ere, S., and Brofferio, S.}
\newblock Internal diffusion limited aggregation on discrete groups having
  exponential growth.
\newblock {\em Probab. Theory Related Fields 137}, 3-4 (2007), 323--343.

\bibitem{BoucheronS.2013CiAn}
{\sc {Boucheron}, S., {Lugosi}, C., and {Massart}, P.}
\newblock {\em Concentration inequalities. A nonasymptotic theory of
  independence}.
\newblock Oxford University Press, 2013.

\bibitem{chenavier2025construction}
{\sc Chenavier, N., Coupier, D., Penner, K., and Rousselle, A.}
\newblock Construction of ergodic {IDLA} forests in $\mathbb{Z}^d$.
\newblock {\em arXiv preprint arXiv:2506.10476\/} (2025).

\bibitem{chenavier2023bi}
{\sc {Chenavier}, N., {Coupier}, D., and {Rousselle}, A.}
\newblock The bi-dimensional directed {IDLA} forest.
\newblock {\em The Annals of Applied Probability 33}, 3 (2023), 2247--2290.

\bibitem{diaconis1991growth}
{\sc Diaconis, P., and Fulton, W.}
\newblock A growth model, a game, an algebra, lagrange inversion, and
  characteristic classes.
\newblock {\em Rend. Sem. Mat. Univ. Pol. Torino 49}, 1 (1991), 95--119.

\bibitem{DLYY}
{\sc {Duminil-Copin}, H., {Lucas}, C., {Yadin}, A., and {Yehudayoff}, A.}
\newblock {Containing internal diffusion limited aggregation.}
\newblock {\em {Electron. Commun. Probab.} 18\/} (2013), 8.
\newblock Id/No 50.

\bibitem{freiberg2023internal}
{\sc Freiberg, U., Heizmann, N., Kaiser, R., and Sava-Huss, E.}
\newblock Internal aggregation models with multiple sources and obstacle
  problems on sierpi{\'n}ski gaskets.
\newblock {\em Journal of Fractal Geometry 11}, 1 (2023), 111--160.

\bibitem{H08}
{\sc Huss, W.}
\newblock Internal diffusion-limited aggregation on non-amenable graphs.
\newblock {\em Electron. Commun. Probab. 13\/} (2008), 272--279.

\bibitem{HS12}
{\sc Huss, W., and Sava, E.}
\newblock Internal aggregation models on comb lattices.
\newblock {\em Electron. J. Probab. 17\/} (2012), no. 30, 21.

\bibitem{jerison2012logarithmic}
{\sc Jerison, D., Levine, L., and Sheffield, S.}
\newblock Logarithmic fluctuations for internal {DLA}.
\newblock {\em J. Amer. Math. Soc. 25}, 1 (2012), 271--301.

\bibitem{JLS13}
{\sc Jerison, D., Levine, L., and Sheffield, S.}
\newblock Internal {DLA} in higher dimensions.
\newblock {\em Electron. J. Probab. 18\/} (2013), No. 98, 14.

\bibitem{JLS14}
{\sc Jerison, D., Levine, L., and Sheffield, S.}
\newblock Internal {DLA} and the {G}aussian free field.
\newblock {\em Duke Math. J. 163}, 2 (2014), 267--308.

\bibitem{JLS14b}
{\sc Jerison, D., Levine, L., and Sheffield, S.}
\newblock Internal {DLA} for cylinders.
\newblock In {\em Advances in analysis: the legacy of {E}lias {M}. {S}tein},
  vol.~50 of {\em Princeton Math. Ser.} Princeton Univ. Press, Princeton, NJ,
  2014, pp.~189--214.

\bibitem{lawler95}
{\sc {Lawler}, G.~F.}
\newblock {Subdiffusive fluctuations for internal diffusion limited
  aggregation.}
\newblock {\em {The Annals of Probability} 23}, 1 (1995), 71--86.

\bibitem{lawler1992internal}
{\sc Lawler, G.~F., Bramson, M., and Griffeath, D.}
\newblock Internal diffusion limited aggregation.
\newblock {\em The Annals of Probability\/} (1992), 2117--2140.

\bibitem{LP10}
{\sc Levine, L., and Peres, Y.}
\newblock Scaling limits for internal aggregation models with multiple sources.
\newblock {\em J. Anal. Math. 111\/} (2010), 151--219.

\bibitem{LS}
{\sc Levine, L., and Silvestri, V.}
\newblock How long does it take for internal {DLA} to forget its initial
  profile?
\newblock {\em Probab. Theory Related Fields 174}, 3-4 (2019), 1219--1271.

\bibitem{L14}
{\sc Lucas, C.}
\newblock The limiting shape for drifted internal diffusion limited aggregation
  is a true heat ball.
\newblock {\em Probab. Theory Related Fields 159}, 1-2 (2014), 197--235.

\bibitem{Shellef}
{\sc {Shellef}, E.}
\newblock {IDLA on the supercritical percolation cluster.}
\newblock {\em {Electron. J. Probab.} 15\/} (2010), 723--740.
\newblock Id/No 24.

\bibitem{S19}
{\sc Silvestri, V.}
\newblock Internal {DLA} on cylinder graphs: fluctuations and mixing.
\newblock {\em Electron. Commun. Probab. 25\/} (2020), Paper No. 61, 14.

\end{thebibliography}

\end{document}